\documentclass{article}

\PassOptionsToPackage{numbers}{natbib}
    \usepackage[final]{neurips_2019}

\usepackage[utf8]{inputenc} 
\usepackage[T1]{fontenc}    
\usepackage{hyperref}       
\usepackage{url}            
\usepackage{booktabs}       
\usepackage{amsmath,amssymb,amsthm}       
\usepackage{nicefrac}       
\usepackage{microtype}      
\usepackage{mathtools}
\usepackage[x11names,table]{xcolor}
\usepackage{enumitem}
\usepackage{subcaption}
\usepackage{xfrac}

\newtheorem{theorem}{Theorem}

\newtheorem{proposition}[theorem]{Proposition}
\newtheorem{corollary}[theorem]{Corollary}
\newtheorem{lemma}[theorem]{Lemma}
\theoremstyle{definition}
\newtheorem{definition}[theorem]{Definition}

\newcommand{\norm}[1]{\left\lVert#1\right\rVert}
\DeclarePairedDelimiterX{\inp}[2]{\langle}{\rangle}{#1, #2}

\newcommand{\argmin}{\operatornamewithlimits{argmin}}
\newcommand{\BV}{\text{BV}} 
\renewcommand{\P}{\mathcal{P}} 
\newcommand{\X}{\mathcal{X}} 
\newcommand{\E}{\mathop{\mathbb{E}}} 
\renewcommand{\L}{L} 
\newcommand{\F}{\mathcal{F}} 
\newcommand{\R}{\mathbb{R}} 
\newcommand{\Z}{\mathbb{Z}} 
\newcommand{\N}{\mathbb{N}} 
\newcommand{\W}{\mathcal{W}} 
\newcommand{\IID}{\stackrel{IID}{\sim}} 
\renewcommand{\hat}{\widehat} 
\renewcommand{\tilde}{\widetilde} 

\DeclareMathOperator{\ess}{ess}

\title{Nonparametric Density Estimation and\\ Convergence of GANs under Besov IPM Losses}

\author{%
  Ananya Uppal\\
  Department of Mathematical Sciences\\
  Carnegie Mellon University\\
  \texttt{auppal@andrew.cmu.edu} \\
   \And
   Shashank Singh\thanks{Now at Google.} \quad Barnab\'as P\'oczos \\
   Machine Learning Department\\
   Carnegie Mellon University\\
   \texttt{\{sss1,bapoczos\}@cs.cmu.edu} \\
}

\begin{document}

\maketitle

\begin{abstract}
We study the problem of estimating a nonparametric probability density under a large family of losses called Besov IPMs, which include, for example, $\L^p$ distances, total variation distance, and generalizations of both Wasserstein and Kolmogorov-Smirnov distances. For a wide variety of settings, we provide both lower and upper bounds, identifying precisely how the choice of loss function and assumptions on the data interact to determine the minimax optimal convergence rate. We also show that linear distribution estimates, such as the empirical distribution or kernel density estimator, often fail to converge at the optimal rate. Our bounds generalize, unify, or improve several recent and classical results. Moreover, IPMs can be used to formalize a statistical model of generative adversarial networks (GANs). Thus, we show how our results imply bounds on the statistical error of a GAN, showing, for example, that GANs can strictly outperform the best linear estimator.
\end{abstract}
\vspace{-2mm}
\section{Introduction}
\vspace{-2mm}
This paper studies the problem of estimating a nonparametric probability density, using an integral probability metric as a loss. That is, given a sample space $\X \subseteq \R^D$, suppose we observe $n$ IID samples $X_1,...,X_n \IID p$ from a probability density $p$ over $\X$ that is unknown but assumed to lie in a regularity class $\P$. We seek an estimator $\hat p : \X^n \to \P$ of $p$, with the goal of minimizing a loss
\vspace{-2mm}
\begin{equation}
    d_\F(p, \hat p(X_1,...,X_n)) := \sup_{f \in \F} \left| \E_{X \sim p} \left[ f(X) \right] - \E_{X \sim \hat p(X_1,...,X_n)} \left[ f(X) \right] \right|,
    \tag{$*$} \label{def:IPM}
\end{equation}
where $\F$, called the \textit{discriminator class}, is some class of bounded, measurable functions on $\X$.

Metrics of the form~\eqref{def:IPM} are called \textit{integral probability metrics} (IPMs), or $\F$-IPMs\footnote{While the name IPM seems most widely used~\citep{muller1997integral,sriperumbudur2010non,bottou2018geometrical,zellinger2019robust}, many other names have been used for these quantities, including \textit{adversarial loss}~\citep{singh2018adversarial,dong2019towards}, \textit{MMD}~\citep{dziugaite2015training}, and \textit{$\F$-distance} or \textit{neural net distance} \citep{arora2017generalization}.}, and can capture a wide variety of metrics on probability distributions by choosing $\F$ appropriately~\citep{muller1997integral}. This paper studies the case where both $\F$ and $\P$ belong to the family of Besov spaces, a large family of nonparametric smoothness spaces that include, as examples, $\L^p$, Lipschitz/H\"older, and Hilbert-Sobolev spaces.
The resulting IPMs include, as examples, $\L^p$, total variation, Kolmogorov-Smirnov, and Wasserstein distances.
We have two main motivations for studying this problem:
\begin{enumerate}[wide]
    \item This problem unifies nonparametric density estimation with the central problem of empirical process theory, namely bounding quantities of the form $d_\F(P, \hat P)$ when $\hat P$ is the empirical distribution $P_n = \frac{1}{n} \sum_{i = 1}^n \delta_{X_i}$ of the data~\citep{pollard1990empirical}. Whereas empirical process theory typically avoids restricting $\P$ and fixes the estimator $\hat P = P_n$, focusing on the discriminator class $\F$, nonparametric density estimation typically fixes the loss to be an $\L^p$ distance, and seeks a good estimator $\hat P$ for a given distribution class $\P$. In contrast, we study how constraints on $\F$ and $\P$ \emph{jointly} determine convergence rates of a number of estimates $\hat P$ of $P$. In particular, since Besov spaces comprise perhaps the largest commonly-studied family of nonparametric function spaces, this perspective allows us to unify, generalize, and extend several classical and recent results in distribution estimation (see Section~\ref{sec:related_work}).
    \item This problem is a theoretical framework for analyzing generative adversarial networks (GANs). Specifically, given a GAN whose discriminator and generator networks encode functions in $\F$ and $\P$, respectively, recent work~\citep{liu2017approximationInGANs,liang2017well,liang2018well,singh2018adversarial} showed that a GAN can be seen as a distribution estimate\footnote{We assume a good optimization algorithm for computing \eqref{eq:theoretical_GAN}, although this is also an active area of research.}
    \vspace{-2mm}
    \begin{equation}
    \hat P = \argmin_{Q \in \P} \sup_{f \in \F} \left| \E_{X \sim Q} \left[ f(X) \right] - \E_{X \sim \tilde P_n} \left[ f(X) \right] \right| = \argmin_{Q \in \P} d_\F \left( Q, \tilde P_n \right),
    \label{eq:theoretical_GAN}
    \end{equation}
    i.e., an estimate which directly minimizes empirical IPM risk with respect to a (regularized) empirical distribution $\tilde P_n$.
    While, in the original GAN model~\citep{goodfellow2014GANs}, $\tilde P_n$ was the empirical distribution $P_n = \frac{1}{n} \sum_{i = 1}^n \delta_{X_i}$ of the data, \citet{liang2017well} showed that, under smoothness assumptions on the population distribution, performance is improved by replacing $P_n$ with a regularized version $\tilde P_n$, equivalent to the instance noise trick that has become standard in GAN training~\citep{sonderby2016amortised,mescheder2018training}.
    We show, in particular, that, when $\tilde P_n$ is a wavelet-thresholding estimate, a GAN based on sufficiently large fully-connected neural networks with ReLU activations learns Besov probability distributions at the optimal rate.
\end{enumerate}
\vspace{-2mm}

\section{Set up and Notation}
\vspace{-2mm}
\label{sec:setup}
    For non-negative real sequences $\{a_n\}_{n \in \N}$, $\{b_n\}_{n \in \N}$, $a_n \lesssim b_n$ indicates $\limsup_{n \to \infty} \frac{a_n}{b_n} < \infty$, and $a_n~\asymp~b_n$ indicates $a_n \lesssim b_n \lesssim a_n$.
    For $p \in [1,\infty]$, $p' := \frac{p}{p - 1}$ denotes the H\"older conjugate of $p$ (with $1' = \infty$, $\infty' = 1$). $\L^p(\R^D)$ (resp. $l^p$) denotes the set of functions $f$ (resp. sequences $a$) with $\norm{f}_p := \left(\int |f(x)|^p \, dx \right)^{1/p}<\infty$ (resp. $\norm{a}_{l^p} := \left(\sum_{n\in \N}|a_n|^p\right)^{1/p}<\infty$).
        \vspace{-2mm}
    \subsection{Multiresolution Approximation and Besov Spaces}
    \vspace{-2mm}
    We now provide some notation that is necessary to define the family of Besov spaces studied in this paper. Since the statements and formal justifications behind these definitions are a bit complex, some technical details are relegated to the Appendix, and several well-known examples from the rich class of resulting spaces are given in Section~\ref{sec:related_work}.
    The diversity of Besov spaces arises from the fact that, unlike the H\"older or Sobolev spaces that they generalize, Besov spaces model functions simultaneously across multiple spatial scales. In particular, they rely on the following notion:
    \begin{definition}
        A \emph{multiresolution approximation (MRA)} of $\L^2(\R^D)$ is an increasing sequence $\{V_j\}_{j\in \Z}$ of closed linear subspaces of $L^2(\R^D)$ with the following properties:
        \begin{enumerate}[wide,noitemsep]
            \item 
                $\bigcap_{j=-\infty}^\infty V_j = \{0\}$, and the closure of $\bigcup_{j=-\infty}^\infty V_j = \L^2(\R^D)$.
            \item 
                For $f\in \L^2(\R^D), k\in \Z^D, j \in \Z$,
                    $f(x) \in V_0 \Leftrightarrow f(x-k)\in V_0$ \& $f(x) \in V_j \Leftrightarrow f(2x)\in V_{j+1}$.
            \item 
                For some ``father wavelet'' $\phi\in V_0$, $\{\phi(x-k) : k\in \Z^D\}$ is an orthonormal basis of $V_0 \subset \L^2(\R^D)$.
        \end{enumerate}
    \end{definition}
    
    For intuition, consider the best-known MRA of $\L^2(\R)$, namely the Haar wavelet basis. Let $\phi(x) = 1_{\{[0,1)\}}$ be the Haar father wavelet, let $V_0 = \text{Span}\{ \phi(x - k) : k \in \Z\}$ be the span of translations of $\phi$ by an integer, and let $V_j$ defined recursively for all $j\in \Z$ by $V_j = \{f(2x) : f(x) \in V_{j-1}\}$ be the set of horizontal scalings of functions in $V_{j-1}$ by $\sfrac{1}{2}$. Then, $\{V_j\}_{j \in \Z}$ is an MRA of $\L^2(\R)$.
    
    The importance of an MRA is that it generates an orthonormal basis of $\L^2(\R^D)$, via the following:
    \begin{lemma}[\citep{meyer1992wavelets}, Section 3.9]
    Let $\{V_j\}_{j \in \Z}$ be an MRA of $\L^2(\R^D)$ with father wavelet $\phi$.
    Then, for $E = \{0,1\}^D\setminus (0,\dots,0)$, there exist ``mother wavelets'' $\{\psi_\epsilon\}_{\epsilon \in E}$ such that
    $\{2^{D j/2}\psi_\epsilon(2^jx-k) : \epsilon\in E, k\in \Z^D\} \cup \{2^{D j/2}\phi(2^jx-k) : k\in \Z^D\}$ is an orthonormal basis of $V_j \subseteq \L^2(\R^D)$.
    \end{lemma}

    

    Let $\Lambda_j = \{ 2^{-j}k + 2^{-j-1}\epsilon: k\in \Z^D, \epsilon\in E\} \subseteq \R^D$.
    Then $k, \epsilon$ are uniquely determined for any $\lambda\in \Lambda_j$. Thus, for all $\lambda\in \Lambda := \bigcup_{j\in\Z}\Lambda_j$, we can let $\psi_\lambda(x) = 2^{D j/2}\psi_{\epsilon}(2^j x-k)$. Equipped with the orthonormal basis $\{\psi_\lambda : \lambda \in \Lambda\}$ of $\L^2(\R^D)$, we are almost ready to define Besov spaces.
    
    For technical reasons (see, e.g., \citep[Section 3.9]{meyer1992wavelets}), we need MRAs of smoother functions than Haar wavelets, which are called \emph{$r$-regular}. Due to space constraints, $r$-regularity is defined precisely in Appendix A; we note here that standard $r$-regular MRAs exist, such as the Daubechies wavelet~\citep{daubechies1992ten}. We assume for the rest of the paper that the wavelets defined above are supported on $[-A,A]$.
    
    \begin{definition}[Besov Space] Let $0 \leq \sigma < r$, and let $p,q \in [1,\infty]$. Given an $r$-regular MRA of $L^2(\R^D)$ with father and mother wavelets $\phi, \psi$ respectively, the \emph{Besov space} $B^{\sigma}_{p,q}(\R^D)$ is defined as the set of functions $f : \R^D \to \R$ such that, the wavelet coefficients
        \vspace{-1mm}
        \[
            \alpha_{k} := \int_{\R^D} f(x)\phi(x-k)dx \;
                \text{ for } \; k \in \Z^D
			    \quad \text{ and } \quad
	        \beta_{\lambda} := \int_{\R^D} f(x) \psi_{\lambda}(x) dx \;
	            \text{ for } \; \lambda \in \Lambda,
		\]
		\vspace{-3mm}
		\[\text{satisfy} \quad\quad
		    \norm{f}_{B_{p,q}^\sigma} := \norm{\{\alpha_k\}_{k\in \Z^D}}_{l^p} +
		    \norm{
		        \left\{
		            2^{j(\sigma+D(1/2-1/p))}
		            \norm{\{\beta_\lambda\}_{\lambda\in \Lambda_j}}_{l^p}
		      \right\}_{j \in \N}}_{l^q}
		    <\infty
		\]
		   The quantity $\|f\|_{B^\sigma_{p,q}}$ is called the \emph{Besov norm of $f$}, and, for any $L > 0$, we write $B^\sigma_{p,q}(L)$ to denote the closed Besov ball $B^\sigma_{p,q}(L) = \{f \in B^\sigma_{p,q} : \|f\|_{B^\sigma_{p,q}} \leq L\}$.
		   When the constant $L$ is unimportant (e.g., for \emph{rates} of convergence), $B^\sigma_{p,q}$ denotes a ball $B_{p,q}^{\sigma}(L)$ of finite but arbitrary radius $L$.
    \end{definition}
    
    
    
    \vspace{-2mm}
    \subsection{Formal Problem Statement}
    \vspace{-2mm}
    Having defined Besov spaces, we now formally state the statistical problem we study in this paper.
    Fix an $r$-regular MRA.
    We observe $n$ IID samples $ X_1,...,X_n \IID p$ from an unknown probability density $p$ lying in a Besov ball $B_{p_g,q_g}^{\sigma_g}(L_g)$ with $\sigma_g<r$.
    We want to estimate $p$, measuring error with an IPM $d_{B_{p_d,q_d}^{\sigma_d}(L_d)}$. Specifically, for general $\sigma_d,\sigma_g,p_d,p_g,q_d,q_g$,
    we seek to bound minimax risk
        \begin{equation}
            M \left( B_{p_d,q_d}^{\sigma_d}, B_{p_g,q_g}^{\sigma_g} \right)
                := \inf_{\hat p} \sup_{p \in B_{p_g,q_g}^{\sigma_g}} \E_{X_{1:n}} \left[ d_{B_{p_d,q_d}^{\sigma_d}} \left( p, \hat p(X_1,\dots,X_n) \right) \right]
                \label{eq:minimax_risk}
        \end{equation}
    of estimating densities in $\F_g = B_{p_g,q_g}^{\sigma_g}$, where the infimum is taken over all estimators $\hat p(X_1, \dots, X_n)$. In the rest of this paper, we suppress dependence of $\hat p(X_1,...,X_n)$ on $X_1,...,X_n$, writing simply $\hat{p}$.
\vspace{-2mm}
\section{Related Work}
\label{sec:related_work}
\vspace{-2mm}
The current paper unifies, extends, or improves upon a number of recent and classical results in the nonparametric density estimation literature. Two areas of prior work are most relevant:
\paragraph{Nonparametric estimation over inhomogeneous smoothness spaces}
    First is the classical study of estimation over inhomogeneous smoothness spaces under $\L^p$ losses. \citet{nemirovskii1985nonparametric} first noticed that, over classes of regression functions with inhomogeneous (i.e., spatially-varying) smoothness, many widely-used regression estimators, called ``linear'' estimators (defined precisely in Section~\ref{sec:linear_estimators}), are provably unable to converge at the minimax optimal rate, in $\L^2$ loss. \citet{donoho1996density} identified a similar phenomenon for estimating probability densities in a Besov space $B_{p_g,q_g}^{\sigma_g}$ on $\R$ under $\L^{p_d'}$ losses with $p_d' > p_g$, corresponding to the case $\sigma_d = 0, D = 1$ in our work. \citep{donoho1996density} also showed that the wavelet-thresholding estimator we consider in Section~\ref{sec:wavelet_upper_bounds} \emph{does} converge at the minimax optimal rate.
    We generalize these phenomena to many new loss functions; in many cases, linear estimators continue to be sub-optimal, whereas the wavelet-thresholding estimator continues to be optimal. We also show that sub-optimality of linear estimators is more pronounced in higher dimensions.
    
    \paragraph{Distribution estimation under IPMs}
    The second, more recent body of results~\citep{liang2017well,singh2018adversarial,liang2018well} concerns nonparametric distribution estimation under IPM losses. Prior work focused on the case where $\F$ and $\P$ are both Sobolev ellipsoids, corresponding to the case $p_d = q_d = p_g = q_g = 2$ in our work. Notably, over these smaller spaces (of homogeneous smoothness), the linear estimators mentioned above are minimax rate-optimal.
    Perhaps the most important finding of these works is that the curse of dimensionality pervading classical nonparametric statistics is significantly diminished under weaker loss functions than $\L^p$ losses (namely, many IPMs). For example, \citet{singh2018adversarial} showed that, when $\sigma_d > D/2$, one can estimate $P$ at the parametric rate $n^{-1/2}$ in the loss $d_{B_{2,2}^{\sigma_d}}$, without \emph{any} regularity assumptions whatsoever on the probability distribution $P$. We generalize this to other losses $d_{B_{p_d,q_d}^{\sigma_d}}$.
    
    These papers were motivated in part by a desire to understand theoretical properties of GANs, and, in particular, \citet{liang2017well} and \citet{singh2018adversarial} helped establish \eqref{eq:theoretical_GAN} as a valid statistical model of GANs. In particular, we note that \citet{singh2018adversarial} showed that the implicit generative modeling problem (``sampling'') in terms of which GANs are usually framed, is equivalent, in terms of minimax convergence rates, to nonparametric density estmation, justifying our focus on the latter problem in this paper.
    We show, in Section~\ref{sec:GAN_upper_bounds}, that, given a sufficiently good optimization algorithm, GANs based on appropriately constructed deep neural networks can learn Besov densities at the minimax optimal rate.
    In this context, our results are among the first to suggest theoretically that GANs can outperform classical density estimators (namely, linear estimators mentioned above).

\citet{liu2017approximationInGANs} provided general sufficient conditions for weak consistency of GANs in a generalization of the model~\eqref{eq:theoretical_GAN}. Since many IPMs, such as Wasserstein distances, metrize weak convergence of probability measures under mild additional assumptions \citet{villani2008optimal}, this implies consistency under these IPMs. However, \citet{liu2017approximationInGANs} did not study \emph{rates} of convergence.


We end this section with a brief survey of known results for estimating distributions under specific Besov IPM losses, noting that our results (Equations~\eqref{eq:general_minimax_rate} and \eqref{eq:linear_minimax_rate} below) generalize all these rates:
\begin{enumerate}[wide,itemsep=0pt]
    \item 
        \textbf{$\L^p$ Distances:} If $\F_d = \L^{p'} = B_{p',p'}^0$, then, for distributions $P, Q$ with densities $p,q \in \L^p$, $d_{\F_d}(P, Q) = \|p - q\|_{\L^p}$. These are the most well-studied losses in nonparametric statistics, especially for $p \in \{1,2,\infty\}$~\citep{nemirovski2000topics,wassermann2006all,tsybakov2009introduction}. \cite{donoho1996density} studied the minimax rate of convergence of density estimation over Besov spaces under $L^p$ losses, obtaining minimax rates
        $n^{-\frac{\sigma_g}{2\sigma_g + D}} + n^{-\frac{\sigma_g + D \left( 1 - 1/p_g - 1/p_d \right)}{2\sigma_g + D \left( 1 - 2/p_g \right)}}$
        over general estimators, and
        $n^{-\frac{\sigma_g}{2\sigma_g + D}} + n^{-\frac{\sigma_g - D/p_g+D/p_d'}{2\sigma_g + D - 2D/p_g+2D/p_d'}}$
        when restricted to linear estimators.
    \item 
        \textbf{Wasserstein Distance:} If $\F_d = C^1(1) \asymp B_{\infty,\infty}^1$ is the space of $1$-Lipschitz functions, then $d_{\F_d}$ is the $1$-Wasserstein or Earth mover's distance (via the Kantorovich dual formulation~\citep{kantorovich1958space,villani2008optimal}). A long line of work has established convergence rates of the empirical distribution to the true distribution in spaces as general as unbounded metric spaces~\citep{weed2017sharp,lei2018convergence,singh2018minimax}). In the Euclidean setting, this is well understood~\citep{dudley1969speed,ajtai1984optimalMatchings,fournier2015rate}, although, to the best of our knowledge, minimax lower bounds have been proven only recently~\citep{singh2018minimax}; this setting intersects with our work in the case $\sigma_d = 1, \sigma_g = 0$, $p_d = \infty$, matching our minimax rate of $n^{-1/D} + n^{-1/2}$.
        More general $p$-Wasserstein distances $W_p$ ($p \geq 1$) cannot be expressed exactly as IPMs, but, our results complement recent results of \citet{weed2019estimation}, who showed that, for densities $p$ and $q$ that are bounded above and below (i.e., $0 < m \leq p,q \leq M < \infty$), the bounds
        $M^{-1/p'} d_{B_{p',\infty}^1}(p, q) \leq W_p(p, q) \leq m^{-1/p'} d_{B_{p',1}^1}(p, q)$
        hold; for such densities, our rates match theirs ($n^{-\frac{1 + \sigma_g}{2\sigma_g + D}} + n^{-1/2}$) up to polylogarithmic factors. \citet{weed2019estimation} showed that, without the lower-boundedness assumption ($m > 0$), minimax rates under $W_p$ are strictly slower (by a polynomial factor in $n$).
        
        In machine learning applications, \citet{arora2017generalization} recently used this rate to argue that, for data from a continuous distribution, Wasserstein GANs~\citep{arjovsky2017wasserstein} cannot generalize at a rate faster than $n^{-1/D}$ (at least without additional regularization, as we use in Theorem~\ref{thm:GAN_upper_bound}).
        A variant in which $\F_d \subset C^1 \cap \L^\infty$ is both uniformly bounded and $1$-Lipschitz gives rise to the Dudley metric~\citep{dudley1972speeds}, which has also been suggested for use in GANs~\citep{abbasnejad2018deep}.
        Finally, we note that the more general distances induced by $\F_d = B_{\infty,\infty}^{\sigma_d}$ have been useful for deriving central limit theorems~\citep[Section 4.8]{chen2010normal}.
    \item 
        \textbf{Kolmogorov-Smirnov Distance:} If $\F_d = \BV \asymp B_{1,\cdot}^1$ is the set of functions of bounded variation, then, in the $1$-dimensional case, $d_{\F_d}$ is the well-known Kolmogorov-Smirnov metric~\citep{daniel1978applied}, and so the famous Dvoretzky–Kiefer–Wolfowitz inequality~\citep{massart1990DKW} gives a parametric convergence rate of $n^{-1/2}$.
    \item 
        \textbf{Sobolev Distances:} If $\F_d = \W^{\sigma_d,2} = B_{2,2}^\sigma$ is a Hilbert-Sobolev space, for $\sigma \in \R$, then $d_{\F_d} = \|\cdot - \cdot\|_{\W^{-\sigma_d,2}}$ is the corresponding negative Sobolev pseudometric~\citep{yosida1995functional}. Recent work~\citep{liang2017well,singh2018adversarial,liang2018well} established a minimax rate of $n^{-\frac{\sigma_g+\sigma_d}{2\sigma_g+1}} + n^{-1/2}$ when $\F_g = \W^{\sigma_g,2}$ is also a Hilbert-Sobolev space.
\end{enumerate}
\vspace{-2mm}
\section{Main Results}
\label{sec:summary_of_results}
\vspace{-2mm}
The three \textbf{main technical contributions} of this paper are as follows:
\begin{enumerate}[wide,noitemsep]
    \item We prove lower and upper bounds (Theorems~\ref{thm:nonlinear_lower_bound} and \ref{thm:nonlinear_upper_bound}, respectively) on minimax convergence rates of distribution estimation under IPM losses when the distribution class $\P = B_{p_g,q_g}^{\sigma_g}$ and the discriminator class $\F = B_{p_d,q_d}^{\sigma_d}$ are Besov spaces; these rates match up to polylogarithmic factors in the sample size $n$. Our upper bounds use the wavelet-thresholding estimator proposed in~\citet{donoho1996density}, which we show converges at the optimal rate for a much wider range of losses than previously known. Specifically, if $M(\F, \P)$ denotes minimax risk~\eqref{eq:minimax_risk}, we show that for $p_d'\geq p_g$, $\sigma_g \geq D/p_g$,
	\vspace{-1mm}
    \begin{equation}
    M \left( B_{p_d,q_d}^{\sigma_d}, B_{p_g,q_g}^{\sigma_g} \right)
        \asymp \max \left\{ n^{-1/2}, n^{-\frac{\sigma_g+\sigma_d}{2\sigma_g+D}}, n^{-\frac{\sigma_g+\sigma_d+D \left( 1 - 1/p_g - 1/p_d \right)}{2\sigma_g + D \left( 1 - 2/p_g \right)}} \right\}.
        \label{eq:general_minimax_rate}
    \end{equation}
    \item We show (Theorem~\ref{thm:linear_minimax_rate}) that, for $p_d'\geq p_g$ and $\sigma_g \geq D/p_g$, no estimator in a large class of distribution estimators, called ``linear estimators'', can converge at a rate faster than
    \vspace{-2mm}
    \begin{equation}
        M_{\text{lin}} \left( B_{p_d,q_d}^{\sigma_d}, B_{p_g,q_g}^{\sigma_g}\right)
        \gtrsim n^{-\frac{\sigma_g+\sigma_d - D/p_g + D/p_d'}{2\sigma_g + D \left( 1 - 2/p_g \right)+2D/p_d'}}.
        \label{eq:linear_minimax_rate}
    \end{equation}
    ``Linear estimators'' include the empirical distribution, kernel density estimates with uniform bandwidth, and the orthogonal series estimators recently used in \citet{liang2017well} and \citet{singh2018adversarial}). 
    The lower bound~\eqref{eq:linear_minimax_rate} implies that, in many settings (discussed in Section~\ref{sec:discussion}), linear estimators converge at sub-optimal rates.
    This effect is especially pronounced when the data dimension $D$ is large and the distribution $P$ has relatively sparse support (e.g., if $P$ is supported near a low-dimensional manifold).
    \item We show that the minimax convergence rate can be achieved by a GAN with generator and discriminator networks of bounded size, after some regularization. As one of the first theoretical results separating performance of GANs from that of classic nonparametric tools such as kernel methods, this may help explain GANs' successes with high-dimensional data such as images.
\end{enumerate}
\vspace{-2mm}
\subsection{Minimax Rates over Besov Spaces}
\label{sec:wavelet_upper_bounds}
\vspace{-2mm}
    We now present our main lower and upper bounds for estimating densities that live in a Besov space under a Besov IPM loss.
	Then, we have the following lower bound on the convergence rate:
		\begin{theorem}(\textbf{Lower Bound})
            Let $r>\sigma_g\geq D/p_g$, then,
            \vspace{-1mm}
            \begin{equation}
                M \left( B_{p_d,q_d}^{\sigma_d}, B_{p_g,q_g}^{\sigma_g} \right)
                \gtrsim \max\left(
                n^{-\frac{\sigma_g+\sigma_d}{2\sigma_g+D}}, \left(\frac{\log n}{n}\right)^{\frac{\sigma_g+\sigma_d+D-D/p_g-D/p_d}{2\sigma_g+D-2D/p_g}}
                \right)
                \label{ineq:lower_bound}
            \end{equation}
            \label{thm:nonlinear_lower_bound}
		\end{theorem}		
		Before giving a corresponding upper bound, we describe the estimator on which it depends.
		
		\textbf{Wavelet-Thresholding:}
		Our upper bound uses the wavelet-thresholding estimator proposed by \citep{donoho1996density}:
		        \begin{align}
					\hat{p}_n
					&=   \sum_{ k\in \Z} \hat{\alpha}_{k} \phi_{k}+
					    \sum_{j=0}^{j_0} \sum_{\lambda\in \Lambda_j} \hat{\beta}_{\lambda}\psi_{\lambda}+
					    \sum_{j= j_0}^{j_1}\sum_{\lambda\in \Lambda_j} \tilde{\beta}_{\lambda} \psi_{\lambda}.
					    \label{eq:wavelet_thresholding_estimator}
				\end{align}
			$\hat p_n$ estimates $p$ via its truncated wavelet expansion, with $\hat{\alpha}_{k} = \frac{1}{n}\sum_{i=1}^n \phi_{k}(X_i)$,
			$\hat{\beta}_{\lambda} = \frac{1}{n}\sum_{i=1}^n \psi_{\lambda}(X_i)$, and
	        $\tilde{\beta}_{\lambda} = \hat{\beta}_{\lambda}\mathbf{1}_{\{\hat{\beta}_{\lambda}>\sqrt{j/n}\}}$ are empirical estimates of respective coefficient of the wavelet expansion of $p$. As \citep{donoho1996density} first showed, attaining optimality over Besov spaces requires truncating high-resolution terms (of order $j \in [j_0,j_1]$) when their empirical estimates are too small; this ``nonlinear'' part of the estimator distinguishes it from the ``linear'' estimators we study in the next section.
	        The hyperparameters $j_0$ and $j_1$ are set to $j_0 = \frac{1}{2\sigma_g+D} \log_2 n$, $j_1 = \frac{1}{2\sigma_g+D-2D/p_g} \log_2 n$.
				
    \begin{theorem}(\textbf{Upper Bound})
        Let $r>\sigma_g\geq D/p_g$ and $p_d'> p_g$. Then, for a constant $C$ depending only on $p_d'$, $\sigma_g$, $p_g$, $q_g$, $D$, $L_g$, $L_d$ and $\norm{\psi_\epsilon}_{p_d'}$,
        \vspace{-1mm}
		\begin{align}
			M \left( B_{p_d,q_d}^{\sigma_d}, B_{p_g,q_g}^{\sigma_g} \right)
				&\leq C \left(
				\sqrt{\log n}\left(n^{-\frac{\sigma_g+\sigma_d}{2\sigma_g+D}}+ n^{-\frac{\sigma_g+\sigma_d-D/p_g+D/p_d'}{2\sigma_g+D-2D/p_g}}
				\right)+n^{-1/2} \right)
			\label{ineq:wavelet_upper_bound}
		\end{align}
		\label{thm:nonlinear_upper_bound}
		\vspace{-2mm}
    \end{theorem}
    We will comment only briefly on Theorems~\ref{thm:nonlinear_lower_bound} and \ref{thm:nonlinear_upper_bound} here, leaving extended discussion for Section~\ref{sec:discussion}. First, note that the lower bound~\eqref{ineq:lower_bound} and upper bound~\eqref{ineq:wavelet_upper_bound} are essentially tight; they differ only by a polylogarithmic factor in $n$. Second, both bounds contain two main terms of interest. The simpler term, $n^{-\frac{\sigma_g+\sigma_d}{2\sigma_g+D}}$, matches the rate observed in the Sobolev case by \citet{singh2018adversarial}. The other term is unique to more general Besov spaces. Depending on the values of $D,\sigma_d,\sigma_g,p_d$, and $p_g$, one of these two terms dominates,
    leading to two main regimes of convergence rates, which we call the ``Sparse'' regime and the ``Dense'' regime. Section~\ref{sec:discussion} discusses these and other interesting phenomena in detail.
\vspace{-2mm}
\subsection{Minimax Rates of Linear Estimators over Besov Spaces}
\label{sec:linear_estimators}
\vspace{-2mm}
    We now show that, for many Besov densities and IPM losses, many widely-used nonparametric density estimators cannot converge at the optimal rate~\eqref{thm:nonlinear_upper_bound}. These estimators are as follows:
    \begin{definition}[Linear Estimator] 
    Let $(\Omega, \F, P)$ be a probability space. An estimate $\hat P$ of $P$ is said to be \emph{linear} if there exist functions $T_i(X_i,\cdot) : \F \to \R$ such that for all measurable $A \in \F$,
    \vspace{-2mm}
    \begin{equation}
        \hat{P}(A) = \sum_{i = 1}^n T_i(X_i,A).
        \label{eq:linear_estimate}
    \end{equation}
    \end{definition}
    Classic examples of linear estimators include the empirical distribution ($T_i(X_i, A) = \frac{1}{n} 1_{\{X_i \in A\}}$, the kernel density estimate ($T_i(X_i,A) = \frac{1}{n} \int_A K(X_i,\cdot)$ for some bandwidth $h > 0$ and smoothing kernel $K : \X \times \X \to \R$) and the orthogonal series estimate ($T_i(X_i, A) = \frac{1}{n} \sum_{j = 1}^J g_j(X_i) \int_A g_j$ for some cutoff $J$ and orthonormal basis $\{g_j\}_{j = 1}^\infty$ (e.g., Fourier, wavelet, or polynomial) of $\L^2(\Omega)$).
    
    \begin{theorem}[Minimax rate for Linear Estimators]
        Suppose $r>\sigma_g\geq D/p_g$,
        \vspace{-1mm}
            \begin{align*}
                M_{\text{lin}} \left( B_{p_d,q_d}^{\sigma_d}, B_{p_g,q_g}^{\sigma_g} \right)
                    := \inf_{\hat{P}_{\text{lin}}} \sup_{p \in \F_g} \E_{X_{1:n}} \left[ d_{\F_d} \left( \mu_p, \hat{P} \right) \right]
                    \asymp n^{-\frac12} + n^{-\frac{\sigma_g+\sigma_d-D/p_g+D/p_d'}{2\sigma_g+D-2D/p_g+2D/p_d'}} + n^{-\frac{\sigma_g+\sigma_d}{2\sigma_g+D}}
            \end{align*}
            where the $\inf$ is over all linear estimates of $p\in \F_g$, and $\mu_p$ is the distribution with density $p$.
            \label{thm:linear_minimax_rate}
    \end{theorem}
    One can check that the above error decays no faster than $n^{-\frac{\sigma_g+\sigma_d+D-D/p_g-D/p_d}{2\sigma_g+D-2D/p_g}}$.
    Comparing with the rate in Theorem~\ref{thm:nonlinear_upper_bound}, this implies that, in certain cases, convergence the rate for linear estimators is strictly slower than that for general estimators; i.e., linear estimators fail to achieve the minimax optimal rate over certain Besov space. We defer detailed discussion of this phenomenon to Section~\ref{sec:discussion}.
\vspace{-2mm}
\subsection{Upper Bounds on a Generative Adversarial Network}
\label{sec:GAN_upper_bounds}
\vspace{-2mm}
Pioneered by \citet{goodfellow2014GANs} as a mechanism for applying deep neural networks to the problem of unsupervised image generation, Generative adversarial networks (GANs) have since been widely applied not only to computer vision~\citep{zhang2017stackgan,ledig2017photo}, but also to such diverse problems and data as machine translation using natural language data~\citep{yang2017machineTranslation}, discovering drugs~\citep{kadurin2017drugan} and designing materials~\citep{sanchez2017optimizing} using molecular structure data, inferring expression levels using gene expression data~\citep{ghasedi2018semi}, and sharing patient data under privacy constraints using electronic health records~\citep{choi2017patientRecords}. Besides the Jensen-Shannon divergence used by~\citep{goodfellow2014GANs}, many GAN formulations have been proposed based on minimizing other losses, including the Wasserstein metric~\citep{arjovsky2017wasserstein,gulrajani2017improved}, total variation distance~\citep{lin2018pacgan}, $\chi^2$ divergence~\citep{mao2017least}, MMD~\citep{li2017mmd}, Dudley metric~\citep{abbasnejad2018deep}, and Sobolev metric~\citep{mroueh2017sobolev}.
The diversity of data types and losses with which GANs have been used motivates studying GANs in a very general (nonparametric) setting.
In particular, Besov spaces likely comprise the largest widely-studied family of nonparametric smoothness class; indeed, most of the losses listed above are Besov IPMs.


GANs are typically described as a two-player minimax game between a generator network $N_g$ and a discriminator network $N_d$; we denote by $\F_d$ the class of functions that can be implemented by $N_d$ and by $\F_g$ the class of distributions that can be implemented by $N_g$. A recent line of work has argued that a natural statistical model for a GAN as a distribution estimator is
\vspace{-1mm}
\begin{equation}
    \hat P := \argmin_{Q \in \F_g} \sup_{f \in \F_d} \E_{X \sim Q} \left[ f(X) \right] - \E_{X \sim \tilde P_n} \left[ f(X) \right],
    \label{eq:GAN_estimate}
\end{equation}
where $\tilde P_n$ is an (appropriately regularized) empirical distribution, and that, when $\F_d$ and $\F_g$ respectively approximate classes $\F$ and $\P$ well, one can bound the risk, under $\F$-IPM loss, of estimating distributions in $\P$ by~\eqref{eq:GAN_estimate}~\citep{liu2017approximationInGANs,liang2017well,singh2018adversarial,liang2018well}.
We emphasize, that, as \citet{singh2018adversarial} showed, the minimax risk in this framework is identical to that under the ``sampling'' (or ``implicit generative modeling''~\citep{mohamed2016learning}) framework in terms of which GANs are usually cast.
\footnote{As in these previous works, we assume implicitly that the optimum~\eqref{eq:GAN_estimate} can be computed; this complex saddle-point problem is itself the subject of a related but distinct and highly active area of work~\citep{nagarajan2017gradient,arjovsky2017towards,liang2018interaction,gidel2018negative}.}

In this section, we show such a result for Besov spaces; namely, we show the existence of a particular GAN (specifically, a sequence of GANs, necessarily growing with the sample size $n$), that estimates distributions in a Besov space at the minimax optimal rate~\eqref{ineq:wavelet_upper_bound} under Besov IPM losses. This construction uses a standard neural network architecture (a fully-connected neural network with rectified linear unit (ReLU) activations), and a simple data regularizer $\tilde P_n$, namely the wavelet-thresholding estimator described in Section~\ref{sec:wavelet_upper_bounds}.
Our results extend those of \citet{liang2017well} and \citet{singh2018adversarial}, for Wasserstein loss over Sobolev spaces, to general Besov IPM losses over Besov spaces.
We begin with a formal definition of the network architectures that we consider:

 \begin{definition}
    A \emph{fully-connected ReLU network} $f_{(A_1,...,A_H),(b_1,...,b_H)} : \R^W \to \R$ has the form
    \[ 
        A_H \eta \left( A_{H-1} \eta \left( \cdots \eta(A_1 x + b_1) \cdots \right) + b_{H - 1} \right) + b_H,
    \]
    where, for each $\ell \in [H - 1]$, $A_\ell \in \R^{W \times W}$, and $A_H \in \R^{1 \times W}$ and the ReLU operation $\eta(x) = \max\{x,0\}$ is applied element-wise to vectors in $\R^W$.
    \end{definition}
    The size of $f_{(A_1,...,A_H),(b_1,...,b_H)}(x)$ can be measured in terms of the following four (hyper)parameters:  the \emph{depth} $H$, the \emph{width} $W$, the \emph{sparsity} $S := \sum_{\ell \in [H]} \|A_\ell\|_{0,0} + \|b_\ell\|_0$ (i.e., the total number of non-zero weights), and the \emph{maximum weight} $B := \max \{\|A_\ell\|_{\infty,\infty}, \|b_\ell\|_\infty : \ell \in [H]\}$.
    For given size parameters $H,W,S,B$ we write $\Phi(H,W,S,B)$ to denote the set of functions satisfying the corresponding size constraints.

Our results rely on a recent construction (Lemma 17 in the Appendix), by \citep{suzuki2018adaptivity}, of a fully-connected ReLU network that approximates Besov functions.
\citep{suzuki2018adaptivity} used this approximation to bound the risk of a neural network for nonparametric regression over Besov spaces, under $\L^r$ loss. Here, we use this approximation result Lemma 17 to bound the risk of a GAN for nonparametric distribution estimation over Besov spaces, under the much larger class of Besov IPM losses. Our precise result is as follows:

\begin{theorem}[Convergence Rate of a Well-Optimized GAN]
    \label{thm:GAN_upper_bound}
    Fix a Besov density class $B_{p_g,q_g}^{\sigma_g}$
    with $\sigma_g > D/p_g$ and discriminator class $B_{p_d,q_d}^{\sigma_d}$ with $\sigma_d>D/p_d$. Then, for any desired approximation error $\epsilon > 0$, one can construct a GAN $\hat p$ of the form~\eqref{eq:GAN_estimate} (with $\tilde{p}_n$)  with discriminator network $N_d \in \Phi(H_d,W_d,S_d,B_d)$ and generator network $N_g \in \Phi(H_g,W_g,S_g,B_g)$, s.t. for all $p \in B_{p_g,q_g}^{\sigma_g}$
    \begin{align*}
        \E \left[ d_{B_{p_d,q_d}^{\sigma_d}} \left( \hat{p}, p \right) \right] \lesssim
        \epsilon + 
        \E
        d_{B^{\sigma_d}_{p_d,q_d}}(\tilde{p}_n,p)
        \end{align*}
   where $H_d$, $H_g$ grow logarithmically with $1/\epsilon$, $W_d,S_d,B_d,W_g,S_g$, $B_g$ grow polynomially with $1/\epsilon$ and $C > 0$ is a constant that depends only on $B_{p_d,q_d}^{\sigma_d}$ and $B_{p_g,q_g}^{\sigma_g}$.
\end{theorem}
    
     This theorem implies that the rate of convergence of the GAN estimate $\hat{p}$ of the form \ref{eq:GAN_estimate} is the same as the convergence rate of the estimator $\tilde{p}_n$ with which the GAN estimate is generated (Here we assume that all distributions have densities).  Therefore, given our upper bound from theorem \ref{thm:nonlinear_upper_bound} we have the following direct consequence.
     
    \begin{corollary}
        For a Besov density class $B_{p_g,q_g}^{\sigma_g}$ with $\sigma_g > D/p_g$ and discriminator class $B_{p_d,q_d}^{\sigma_d}$ with $\sigma_d>D/p_d$ there exists an appropriately constructed GAN estimate $\hat{p}$ s.t.
        \[
            d_{\F_d}(\hat{p},p)\leq 
                \left(
                    n^{-\eta(D,\sigma_d,p_d,\sigma_g,p_g)} \sqrt{\log n} 
                \right)
        \]
        where $\eta(D,\sigma_d,p_d,\sigma_g,p_g) = \min \left\{ \frac{1}{2}, \frac{\sigma_g+\sigma_d}{2\sigma_g+D}, \frac{\sigma_g+\sigma_d + D - D/p_g - D/p_d'}{2\sigma_g + D \left( 1 - 2/p_g \right)} \right\}$ is the exponent from~\eqref{ineq:wavelet_upper_bound}.
    \end{corollary}
In other words there is a GAN estimate that is minimax rate optimal for a smooth class of densities over an IPM generated by a smooth class of discriminator functions.

\section{Discussion of Results}
\label{sec:discussion}
\vspace{-2mm}
In this section, we discuss some general phenomena that can be gleaned from our technical results.

First, we note that, perhaps surprisingly, $q_d$ and $q_g$ do not appear in our bounds. \citet{tao2011functionSpaces} suggests that $q_d$ and $q_g$ may have only logarithmic effects (contrasted with the polynomial effects of $\sigma_d$, $p_d$, $\sigma_g$, and $p_g$). Thus, a more fine-grained analysis to close the polylogarithmic gap between our lower and upper bounds for general estimators (Theorems~\ref{thm:nonlinear_lower_bound} and \ref{thm:nonlinear_upper_bound}) might require incorporating $q_d$ and $q_g$.

On the other hand, the parameters $\sigma_d$, $p_d$, $\sigma_g$, and $p_g$ each play a significant role in determining minimax convergence rates, in both the linear and general cases. We first discuss each of these parameters independently, and then discuss some interactions between them.
\vspace{-2mm}
\paragraph{Roles of the smoothness orders $\sigma_d$ and $\sigma_g$}
As a visual aid for understanding our results, Figure~\ref{fig:phase_diagrams} show phase diagrams of minimax convergence rates, as functions of discriminator smoothness $\sigma_d$ and distribution smoothness $\sigma_g$, in the illustrative case $D = 4$, $p_d = 1.2$, $p_g = 2$. When $1/p_g + 1/p_d > 1$, a minimum total smoothness $\sigma_d + \sigma_g \geq D(1/p_d + 1/p_g - 1)$ is needed for consistent estimation to be possible -- this fails in the ``Infeasible'' region of the phase diagrams. Intuitively, this occurs because $\F_d$ is not contained in the topological dual $\F_g'$ of $\F_g$.
For linear estimators, even greater smoothness $\sigma_d + \sigma_g \geq D(1/p_d + 1/p_g)$ is needed. At the other extreme, for highly smooth discriminator functions, both linear and nonlinear estimators converge at the parametric rate $O \left( n^{-1/2} \right)$, corresponding to the ``Parametric'' region. In between, rates for linear estimators vary smoothly with $\sigma_d$ and $\sigma_g$, while rates for nonlinear estimators exhibit another phase transition on the line $\sigma_g + 3\sigma_d = D$; to the left lies the ``Sparse'' case, in which estimation error is dominated by a small number of large errors at locations where the distribution exhibits high local variation; to the right lies the ``Dense'' case, where error is relatively uniform on the sample space.

The left boundary $\sigma_d = 0$ corresponds to the classical results of \citet{donoho1996density}, who consequently identified the ``Infeasible'', ``Sparse'', and ``Dense'' phases, but not the ``Parametric'' phase. When restricting to linear estimators, the ``Infeasible'' region grows and the ``Parametric'' region shrinks.

\begin{figure}
	\centering
	\begin{subfigure}[t]{0.33\linewidth}
		\centering
		\includegraphics[width=\linewidth]{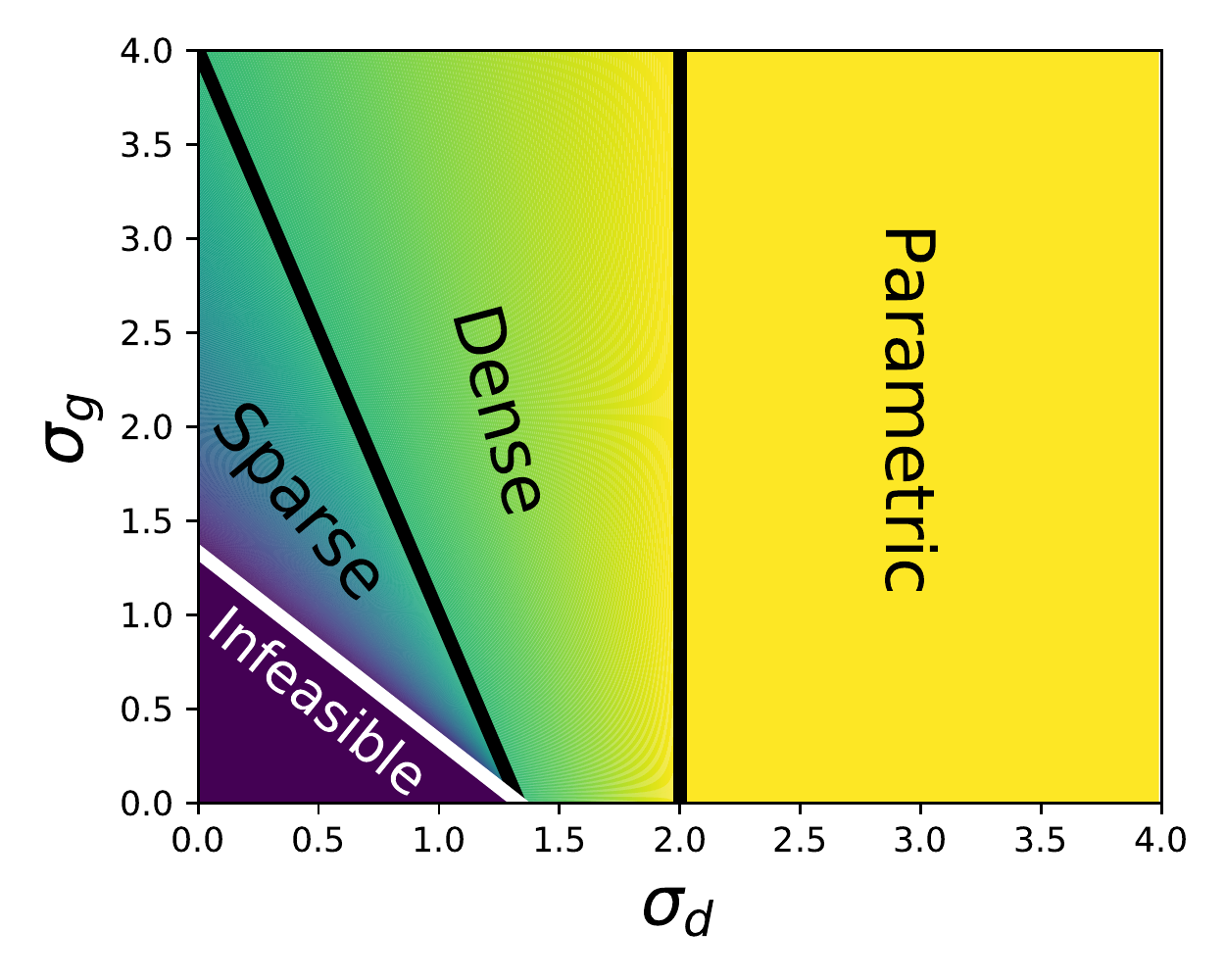}
		\vspace{-3mm}
		\caption{General Estimators}\label{fig:nonlinear_phase_diagram}
	\end{subfigure}~
	\begin{subfigure}[t]{0.66\linewidth}
		\centering
		\includegraphics[width=\linewidth]{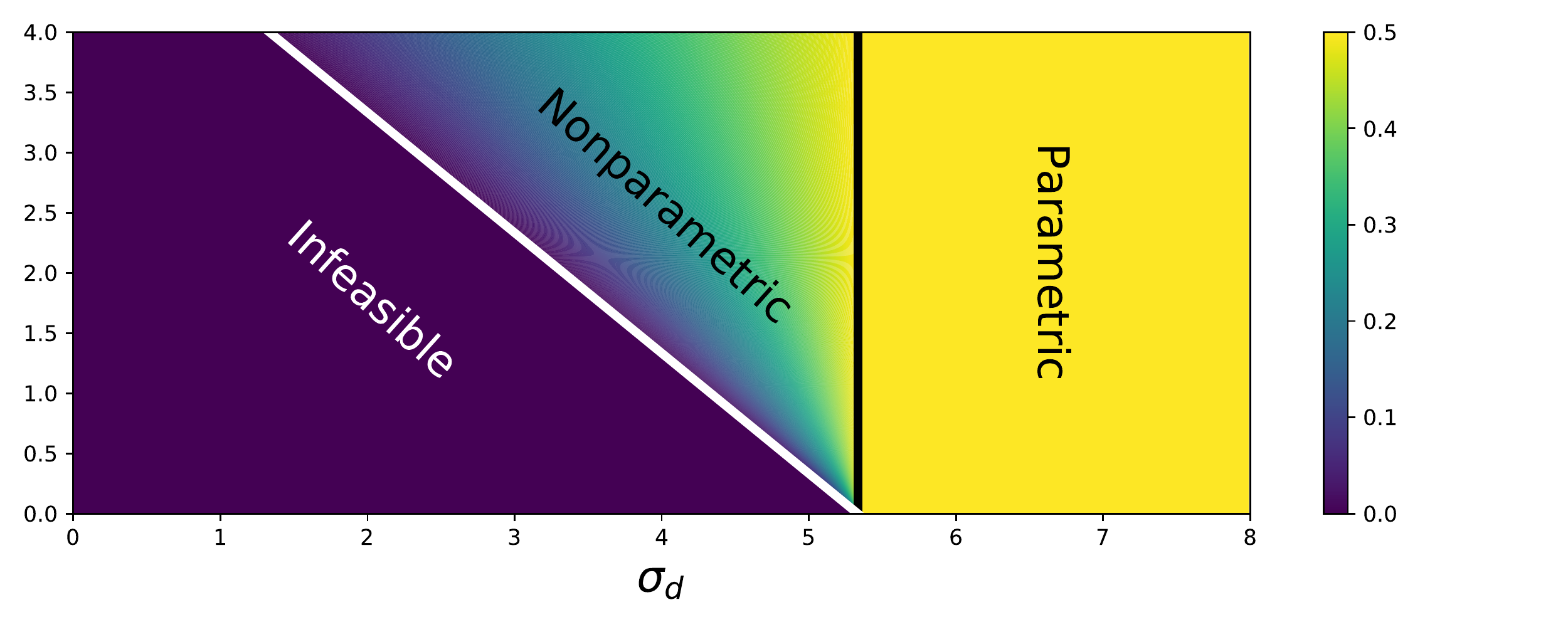}
		\vspace{-3mm}
		\caption{Linear Estimators}\label{fig:linear_phase_diagram}
	\end{subfigure}
	\caption{Minimax convergence rates as functions of discriminator smoothness $\sigma_d$ and distribution function smoothness $\sigma_g$, for (a) general and (b) linear estimators, in the case $D = 4$, $p_d = 1.2$, $p_g = 2$. Color shows exponent of minimax convergence rate (i.e., $\alpha(\sigma_d, \sigma_g)$ such that $M \left( B_{1.2,q_d}^{\sigma_d}(\R^D), B_{2,q_g}^{\sigma_g}(\R^D) \right) \asymp n^{-\alpha(\sigma_d, \sigma_g)}$), ignoring polylogarithmic factors.}\label{fig:phase_diagrams}
	\vspace{-3mm}
\end{figure}

\vspace{-2mm}
\paragraph{Role of the powers $p_d$ and $p_g$}
At one extreme ($p_d = \infty$) lie $\L^1$ or total variation loss ($\sigma_d = 0$), Wasserstein loss ($\sigma_d = 1$), and its higher-order generalizations, for which we showed the rate
\vspace{-2mm}
\[M \left( B_{\infty,q_d}^{\sigma_d}, B_{p_g,q_g}^{\sigma_g} \right) \asymp n^{-\frac{\sigma_g + \sigma_d}{2\sigma_g + D}} + n^{-1/2},\]
generalizing the rate first shown by \citet{singh2018adversarial} for Hilbert-Sobolev classes to other distribution classes, such as $\F_g = \BV$. Because discriminator functions in this class exhibit homogeneous smoothness, these losses effectively weight the sample space relatively uniformly in importance, the ``Sparse'' region in Figure~(\ref{fig:nonlinear_phase_diagram}) vanishes, and linear estimators can perform optimally.

At the other extreme ($p_d = 1$) lie $\L^\infty$ loss ($\sigma_d = 0$), Kolmogorov-Smirnov loss ($\sigma_d = 1$), and its higher-order generalizations,
for which we have shown that the rate is always
\vspace{-1mm}
\[M \left( B_{1,q_d}^{\sigma_d}, B_{p_g,q_g}^{\sigma_g} \right) \asymp n^{-\frac{\sigma_g + \sigma_d + D(1 - 1/p_d - 1/p_g)}{2\sigma_g + D(1 - 2/p_g)}} + n^{-1/2};\]
except in the parametric regime ($D \leq 2 \sigma_d$), this rate differs from that of \citet{singh2018adversarial}. Because discriminator functions can have inhomogeneous smoothness, and hence weight some portions of the sample space much more heavily than others, the ``Dense'' region in Figure~\ref{fig:nonlinear_phase_diagram} vanishes, and linear estimators are always sub-optimal.
We note that \citet{sadhanala2018higherOrderKS} recently proposed using these higher-order distances (integer $\sigma_d > 1$) in a fast two-sample test that generalizes the well-known Kolmogorov-Smirnov test, improving sensitivity to the tails of distributions; our results may provide a step towards understanding theoretical properties of this test.

\paragraph{Comparison of linear and general rates}
Letting $\sigma_g' := \sigma_g - D(1/p_g+1/p_d)$, one can write the sparse term of the linear minimax rate in the same form as the Dense rate, replacing $\sigma_g$ with $\sigma_g'$:
\vspace{-2mm}
\begin{equation}
    M_{\text{lin}} \left( B_{p_d,q_d}^{\sigma_d}, B_{p_g,q_g}^{\sigma_g} \right)
        \asymp n^{-\frac{\sigma_g'+\sigma_d}{2\sigma_g' + D}}.
    \label{rate:linear_rate_rewritten}
\end{equation}
This is not a coincidence; Morrey's inequality~\citep[Section 5.6.2]{evans2010partial} in functional analysis tells us that for general $\sigma_g > D(1/p_g+1/p_d)$, $\sigma_g' := \sigma_g - D(1/p_g+1/p_d)$ is largest possible value such that the embedding $B_{p_g,p_g}^{\sigma_g} \subseteq B_{p_d,p_d}^{\sigma_g'}$ holds. In the extreme case $p_d = \infty$ (corresponding to generalizations of total variation loss), one can interpret the rate~\eqref{rate:linear_rate_rewritten} as saying that linear estimators benefit only from homogeneous (e.g., H\"older) smoothness, and not from weaker inhomogeneous (e.g., Besov) smoothness. For general $p_d$, linear estimator can still benefit from inhomogeneous smoothness, but to a lesser extent than general minimax optimal estimators.


\paragraph{Conclusions}
We have shown, up to log factors, unified minimax convergence rates for a large class of pairs of $\F_d$-IPM losses and distribution classes $\F_g$. By doing so, we have generalized several phenomena that had observed in special cases previously.
First, under sufficiently weak loss functions, distribution estimation is possible at the parametric rate $O(n^{-1/2})$ even over very large nonparametric distribution classes.
Second, in many cases, optimal estimation requires estimators that adapt to inhomogeneous smoothness conditions; many commonly used distribution estimators fail to do this, and hence converge at sub-optimal rates, or even fail to converge.
Finally, GANs with sufficiently large fully-connected ReLU neural networks using wavelet-thresholding regularization perform statistically minimax rate-optimal distribution estimation over inhomogeneous nonparametric smoothness classes (assuming the GAN optimization problem can be solved accurately). Importantly, since GANs optimize IPM losses much weaker than traditional $\L^p$ losses, they may be able to learn reasonable approximations of even high-dimensional distributions with tractable sample complexity, perhaps explaining why they excel in the case of image data. Thus, our results suggest that the curse of dimensionality may be less severe than indicated by classical nonparametric lower bounds.

\appendix
\section{Technical Definitions and Notation}
   
    As noted in the main text, we need a multiresolution approximation (MRA) satisfying an $r$-regularity condition, defined as follows:
    \begin{definition}
        Given a non-negative integer $r$, an MRA is called \emph{$r$-regular} if the function $\phi$ can be chosen in such a way that, for every $m\in \N$ and multi-index $\alpha = (\alpha_1, \dots, \alpha_D) \in \N^D$ satisfying $|\alpha|\leq r$, for some constant $C_{\alpha,m}$,
                $|\partial^{\alpha}\phi(x)| \leq C_{\alpha,m} (1+|x|)^{-m}$.
        Here, $\partial^{\alpha} = (\partial/\partial x_1)^{\alpha_1}\cdots (\partial/\partial x_D)^{\alpha_D}$ is the mixed derivative of index $\alpha$, $|\alpha| = \sum_{j=1}^D \alpha_j$ and $|x|$ is any of the equivalent norms on a finite dimensional Euclidean space. That is, all derivatives of $\phi$ of order up to $r$ are bounded and decay at a rate faster than any polynomial.
        \label{def:regularity}
    \end{definition}
    
    While constructing an $r$-regular MRA is nontrivial, it suffices for our purpose to note that $r$-regular MRAs exist;
    the most famous example is the Daubechies wavelet~\citep{daubechies1992ten,meyer1992wavelets}.
    
    We also note the following result showing that for any function in $V_j$ (i.e., at a certain ``level'' in the MRA) its $\L^p$ norm is equivalent to the $l^p$ sequence norm of its coefficients in the wavelet basis; this helps motivate the sequence-based definition of the Besov norm.
    \begin{proposition}[\citet{meyer1992wavelets}, Section 6.10, Proposition 7]\label{lemma1}
        There exist positive constants $C, C'$ s.t. for every $1\leq p \leq \infty$, $j\in \Z$ and $\{\alpha_k\}\in l^p$, $f(x) = \sum a_k 2^{D j/2}\psi_\epsilon(2^jx-k)$, $\epsilon\in E,k\in \Z^D$,
        \vspace{-2mm}
         \[
            C\norm{f}_p \leq 2^{D j(1/2-1/p)}
                \left(
                    \sum |a_k|^p
                    \right)^{1/p}
                \leq C'\norm{f}_p.
         \]
        \label{norm_prop}
    \end{proposition}
    Appendix A.1 of \citet{donoho1996density} offers a more extended background of Besov spaces, including how the sequence-based definition corresponds to more conventional smoothness measures (moduli of continuity), as well as some direct connections between Besov spaces and minimax theory for linear estimators.
    
\section{ Upper Bound - Linear Case}
        For any density function $p$ let 
            \begin{align*}
                \alpha_{ k}^p &= \int \phi_k(x)p(x)dx\\
                \beta_{\lambda}^p &= \int \psi_{\lambda}(x)p(x)dx
            \end{align*}
        
        We first show that Besov IPMs essentially measure the distance in co-efficient space between compactly supported densities.
        
        \begin{lemma}
            \label{lemma:coeff}
            For any compactly supported probability densities $p$, $q\in \L_{p_d'}$ where $\F_d = B^{\sigma_d}_{p_d,q_d}$
            \[
                d_{\F_d}(p,q) =
                    \sup_{f\in \F_d}\left| 
					\sum_{k\in \Z}
					\alpha^f_{ k}
					\left(\alpha_{ k}^p- \alpha_{ k}^q\right) + 
					\sum_{j\geq 0}\sum_{\lambda\in \Lambda} 
					\beta^f_{\lambda}
					\left( \beta_{\lambda}^p- \beta_{\lambda}^q\right)\right|
            \]
            where for $f\in \F_d$
                \[
			    f = \sum_{k\in \Z}\alpha^f_{k}\phi_{k} + 
					\sum_{j\geq 0}\sum_{\lambda\in \Lambda_j} \beta^f_{\lambda}\psi_{\lambda}
			    \]
        \end{lemma}
        \begin{proof}
            We notice that the convergence to $f$ above is in the $\L_\infty$ norm. So for probability measures $P, Q$ we have, 
			\begin{align*}
				d_{\F_d}(p,q) 
					&= \sup_{f\in \F_d}|\mathrm{E}_{X\sim p}[f(X)]-\mathrm{E}_{X\sim q}[f(X)]|\\
					&= \sup_{f\in \F_d}\left|\int_{\mathcal{X}} f(x)p(x)dx- f(x)q(x)dx\right|\\
					&= \sup_{f\in \F_d}\left|\int_{\mathcal{X}}
					\left(
					\sum_{k\in \Z}\alpha^f_{k}\phi_{k}(x) + 
					\sum_{j\geq 0}\sum_{\lambda\in \Lambda_j} \beta^f_{\lambda}\psi_{\lambda}(x)
					\right)
					\left( p(x)- q(x)\right)dx\right|
			\end{align*}
			If $p, q$ are compactly supported on $[-B,B]$ then we can assume WLOG that $f$ is compactly supported on $[-B,B]$ so convergence of $f_n$ to $f$ in $\L^\infty$ norm implies convergence in $\L^1$ norm.
			Therefore, 
			\begin{align*}
				d_{\F_d}(P,Q) 	
					&= \sup_{f\in \F_d}\left| 
					\sum_{k\in \Z}
					\int_{\mathcal{X}}\alpha^f_{ k}\phi_{k}
					\left( dP(x)- dQ(x)\right) + 
					\sum_{j\geq 0}\sum_{\lambda\in \Lambda_j} 
					\int_{\mathcal{X}}
					\beta^f_{\lambda}\psi_{\lambda}
					\left( dP(x)- dQ(x)\right)\right|\\
					&= \sup_{f\in \F_d}\left| 
					\sum_{k\in \Z}
					\alpha^f_{ k}
					\left(\alpha_{ k}^p- \alpha_{ k}^q\right) + 
					\sum_{j\geq 0}\sum_{\lambda\in \Lambda} 
					\beta^f_{\lambda}
					\left( \beta_{\lambda}^p- \beta_{\lambda}^q\right)\right|
			\end{align*}
        \end{proof}
        
    We will need the following inequalities to estimate the error of the wavelet estimator under the IPM loss. 
    
    The first lemma is the standard upper bound on the $m$th moment of a sum of IID random variables with bounded variance. The second is a standard concentration inequality used to bound large deviations in our error estimate. 
    
    \begin{lemma}({\bf Rosenthal's Inequality} (\cite{rosenthal1970subspaces}))
        Let $m\in \R$ and $Y_1,\dots, Y_n$ be IID random variables with $\E[Y_i]=0$, $\E[Y_i^2]\leq \sigma^2$. Then there is a constant $c_m$ that depends only on $m$ s.t. 
	        \begin{align*}
	            \E\left[
	                \left|\frac{1}{n}\sum_{i=1}^n Y_i
	                \right|^m
	               \right] 
	               & \leq c_m
	                    \left(
	                        \frac{\sigma^m}{n^{m/2}} + \frac{\E|Y_1|^m}{n^{m-1}}
	                    \right)  &\text{for }2<m<\infty,\\
                \E\left[
	                \left|\frac{1}{n}\sum_{i=1}^n Y_i
	                \right|^m
	               \right] 
	                &\leq \sigma^m
	                    n^{-m/2}  &\text{for } 1\leq m\leq 2.
	        \end{align*}
    
    \end{lemma}
    \begin{lemma} ({\bf Bernstein's Inequality} (\cite{bernstein1964modification}))
                If $Y_1, \dots, Y_n$ are IID random variables such that $\E[Y_i]=0$, $\E[Y_i^2]=\sigma^2$ and $|Y_i|\leq \norm{Y}_\infty<\infty$, then 
                    \[
                        \Pr\left(\left|
                        \frac{1}{n}\sum_{i=1}^n Y_i\right|>\lambda\right)
                            \leq 2\exp\left( -\frac{n\lambda^2}{2(\sigma^2+\norm{Y}_\infty\lambda/3)}
                            \right)
                    \]
        where $\norm{Y}_\infty = \ess \sup Y$.
    \end{lemma}
        
        Given discriminator and generator classes as
        \begin{align*}
            \F_d &= \{f : \norm{f}^{\sigma_d}_{p_d, q_d}\leq L_d\}\\
            \F_g &= \{p : \norm{p}^{\sigma_g}_{p_g, q_g}\leq L_g\}\cap \P\\
            \P &= \{p:p\geq 0, \norm{p}_{\L^1}=1, \text{supp}(p)\subseteq [-T,T]\},
        \end{align*}
        we decompose $f \in \F_d$ as
    		\[
			    f = \sum_{k\in \Z}\alpha_{k}\phi_{k} + 
					\sum_{j\geq 0}\sum_{\lambda\in \Lambda_j} \beta_{\lambda}\psi_{\lambda}.
			\]
	    We use the linear wavelet estimator to demonstrate the upper bound. 
		Let $X_1, \dots, X_n$ be IID with density $ p\in \F_g$ and consider the wavelet estimator of $p$ i.e. 
				\begin{align*}
				    p
					&=   \sum_{ k\in \Z} \alpha^p_{k} \phi_{k}+
					    \sum_{j\geq 0} \sum_{\lambda\in \Lambda_j} \beta^p_{\lambda}\psi_{\lambda}\\
					\hat{p}_n 
					&=   \sum_{ k\in \Z} \hat{\alpha}_{k} \phi_{k}+
					    \sum_{j=0}^{j_0} \sum_{\lambda\in \Lambda_j} \hat{\beta}_{\lambda}\psi_{\lambda}
				\end{align*}
			where 
			\[
			    \begin{aligned}[c]
			        \alpha^p_{k} &=  \E_{X\sim p}[\phi_{k}(X)]\\
			        \beta^p_{\lambda} &= \E_{X\sim p}[\psi_{\lambda}(X)]
			    \end{aligned}\hspace{3ex}\qquad
			    \begin{aligned}[c]
			        \hat{\alpha}_{k} &= \frac{1}{n}\sum_{i=1}^n \phi_{k}(X_i)\\
			        \hat{\beta}_{\lambda} &= \frac{1}{n}\sum_{i=1}^n \psi_{\lambda}(X_i)
			    \end{aligned}
			 \]
			 
	    Then applying lemma \ref{lemma:coeff}, we bound
			\begin{align*}
				d_{\F_d}(p,\hat{p}_n) 
					\leq \hspace{5ex}
						&\sup_{f\in \F_d}\sum_{k\in \Z}
						\alpha_{k}
						\left(\alpha_{ k}^p- \hat{\alpha}_{k}\right) 
						&+&
						\sup_{f\in \F_d}
						\sum_{j= 0}^{j_0}\sum_{\lambda\in \Lambda_j} 
						\beta_{\lambda}
						\left( \beta_{\lambda}^p- \hat{\beta}_{\lambda}\right)\\
						&
						&+& 
						\sup_{f\in \F_d}
						\sum_{j\geq j_1}\sum_{\lambda\in \Lambda_j} 
						\beta_{\lambda}
						 \beta_{\lambda}^p
			\end{align*}
		where the first two terms constitute the stochastic error and the last term is the bias. We bound these separately below. 
		We first prove a few lemmas that will be used repeatedly to upper bound the different terms. 
            
        \begin{lemma}
		    \label{lemma:discriminator}
		    Let $n_1, n_2\in \N \cup \{\infty\}$ and $\eta$ be any sequence of numbers. Then
		    
		    \[
		        \E_{X_1,\dots, X_n} \sup_{f\in \F_d}\sum_{j=n_1}^{n_2}\sum_{\lambda\in \Lambda_j} \gamma_{\lambda}\eta_{\lambda}
		        \leq 
		            L_D
			    \sum_{j=n_1}^{n_2}2^{-j(\sigma_d+D/2-D/p_d)}
			    \left(\E_{X_1,\dots, X_n}
			    \sum_{\lambda\in \Lambda_j}|\eta_{\lambda}|^{p_d'}
			    \right)^{1/p_d'}
		    \]
		    Note that if the above is true also if $\gamma = \alpha^f$ and $n_1=n_2=0$.
		\end{lemma}
		\begin{proof}
        Since $f\in \F_d$, applying H\"older's inequality twice we get,
			\begin{align*}
			    \E_{X_1,\dots, X_n} \sup_{f\in \F_d}\sum_{j=n_1}^{n_2}\sum_{\lambda\in \Lambda_j} \gamma_{\lambda}\eta_{\lambda}
			    &\leq  
			        \E_{X_1,\dots, X_n} \sup_{f\in \F_d}\sum_{j=n_1}^{n_2}
			        \norm{\gamma}_{p_d}\norm{\eta}_{p_d'}\\
			    &\leq  
			        \E_{X_1,\dots, X_n} 
			        \sup_{f\in \F_d}
			        \left(\sum_{j=n_1}^{n_2}
			        \left(
    			    2^{j(\sigma_d+D/2-D/p_d)}\norm{\gamma}_{p_d}
    			    \right)^{q_d}\right)^{1/q_d}\\
    			   &\hspace{10ex}\times\sum_{j=n_1}^{n_2}2^{-j(\sigma_d+D/2-D/p_d)}\norm{\eta}_{p_d'}\quad (l^1\subseteq l^{q_d'} )\\
			    &\leq  
			    L_D
			    \sum_{j=n_1}^{n_2}2^{-j(\sigma_d+D/2-D/p_d)} \E_{X_1,\dots, X_n}\norm{\eta}_{p_d'}\\
			    &\leq  
			    L_D
			    \sum_{j=n_1}^{n_2}2^{-j(\sigma_d+D/2-D/p_d)}
			    \left(\E_{X_1,\dots, X_n}
			    \sum_{\lambda\in \Lambda_j}|\eta_{\lambda}|^{p_d'}
			    \right)^{1/p_d'}
			\end{align*}
		where $p_d'$ is the conjugate of $p_d$ i.e. $\frac{1}{p_d}+\frac{1}{p_d'}=1$ and we applied Jensen's to get the last inequality.

        \end{proof}
        \begin{lemma}
            \label{lemma:besov_ub}
            Let $f\in B^{\sigma_g}_{p_g,q_g}$ where $\sigma_g> D/p_g$ then 
            \[
                \norm{f}_\infty \leq 4A\norm{\psi}_\infty L_g (1-2^{(\sigma_g-D/p_g)q_g'})^{-1/q_g'}
            \]
            This implies that sufficiently smooth Besov spaces $B^{\sigma_g}_{p_g,q_q}$ are uniformly bounded.
        \end{lemma}
        \begin{proof}
           We have that 
            $\sum_{k\in \Z^D}\alpha_{ k}\phi_k + 
			\sum_{j\geq 0}\sum_{\lambda\in \Lambda_j} \beta_{\lambda}\psi_{\lambda}$ converges to $f$ in $\L_\infty$. 
			So, using the fact that $l^{p_d}\subseteq l^\infty$ and proposition \ref{lemma1},
				\begin{align*}
					\norm{f}_\infty	\leq 
							2A\norm{\psi}_\infty
							\left(\norm{\{\alpha_k\}_{k\in \Z^D}}_\infty	+
							\sum_{j\geq 0} 2^{D j/2}\norm{\{\beta_{\lambda}\}_{\lambda\in \Lambda_j}}_\infty)\right).
				\end{align*}
				
				We can upper bound, by H\"older's inequality,
				\begin{align*}
				    \sum_{j\geq 0} 2^{D j/2}\norm{\{\beta_{\lambda}\}_{\lambda\in \Lambda_j}}_\infty
						&\leq 	
							\sum_{j\geq 0}\frac{1}{2^{j(\sigma_g-D/p_g)}}
							\times
						 2^{ j(\sigma_g+D/2-D/p_g)}\norm{\{\beta_{\lambda}\}_{\lambda\in \Lambda_j}}_\infty \\
						&\leq 
						     \left(\sum_{j\geq 0}\frac{1}{2^{j(\sigma_g-D/p_g)q_g'}}\right)^{1/q_g'}
							\left(\sum_{j\geq 0} 2^{j q_g(\sigma_g+D/2-D/p_g)}\norm{\{\beta_{\lambda}\}_{\lambda\in \Lambda_j}}_\infty^{q_g} \right)^{1/q_g}
							\\
						&\leq     \left(\frac{1}{1-2^{-(\sigma_g-D/p_g)q_g'}}\right)^{1/q_g'}
							\left(\sum_{j\geq 0} 2^{j q_g(\sigma_g+D/2-D/p_g)}\norm{\{\beta_{\lambda}\}_{\lambda\in \Lambda_j}}_{p_g}^{q_g} \right)^{1/q_g}\\
						&\leq 
						    \left(1-2^{-(\sigma_g-D/p_g)q_g'}\right)^{-1/q_g'} \norm{f}^{\sigma_g}_{p_g q_g} \\
						&\leq 
							\left(1-2^{-(\sigma_g-D/p_g)q_g'}\right)^{-1/q_g'} L_g.
				\end{align*}
				
			Putting the above together we obtain the required upper bound.
        \end{proof}
        
        We also need a few preliminary results namely, the moments of error of linear estimates of the wavelet coefficients are essentially bounded by $1/\sqrt{n}$ and the probability that this error is large is negligibly small. In particular,

        \begin{lemma}(Moment Bounds)
            Let $X_1, \dots, X_n \sim p$, $m\geq 1$ s.t. there is a constant $c$ with $\E_p|\psi_\lambda(X)|^m\leq c 2^{Dj(m/2-1)}$. Let
                \begin{align*}
					\gamma_{\lambda}^p 
					    &= \E[\psi_\lambda(X)],\\
					\hat{\gamma}_{\lambda} 
					    &= \frac{1}{n}\sum_{i=1}^n \psi_\lambda(X_i),
				\end{align*}
		    Then for all $j$ s.t. $2^{Dj}\in \mathcal{O}(n)$, 
				\[
				    \E[|\hat{\gamma}_{jk}-\gamma_{jk}|^m]
				        \leq c n^{-m/2}.
				\]
            where $c = c_m \left(\E_p|\psi_\lambda(X)|^2\right)^{m/2}$ is a constant.
        \end{lemma}
        
        \begin{proof}
            Since $\psi_\lambda$ is bounded for every $\lambda$, let
			\begin{align*}
			        Y_i 
				    &= \psi_\lambda(X_i)- \E[\psi_\lambda(X)]
			\end{align*}
			then for all $m\geq 1$, applying Jensen's inequality repeatedly we get
				\begin{align*}
					\E[|Y_i|^m] 
					    &\leq \E[\left(|\psi_\lambda(X_i)|+|\E[\psi_\lambda(X_i)]|\right)^m] &\text{(triangle inequality)}\\
					    &\leq 2^{m-1}\left( \E[|\psi_\lambda(X_i)|^m]+|\E[\psi_\lambda(X_i)]|^m
					        \right) 
					        &\text{(Jensen's)}\\
						&\leq 
							2^{m}\E[|\psi_\lambda(X_i)|^m].  &\text{(Jensen's)} 
				\end{align*}
            Therefore, by Rosenthal's inequality we have,
			\begin{align*}
	            \E[|\gamma_{\lambda}^p-\hat{\gamma}_{\lambda}|^m]
					&\leq 
			         c_m \left( \left(\E_p|\psi_\lambda(X)|^2\right)^{m/2}
			                	+ c\left(\frac{2^{Dj}}{n}\right)^{(m/2-1)_+}
						\right)n^{-m/2}
			\end{align*}
			where $c_m$ is a constant that only depends on $m$. Therefore,
			\[
			    \E[|\gamma_{\lambda}^p-\hat{\gamma}_{\lambda}|^m]
					\leq c_m \left(\E_p|\psi_\lambda(X)|^2\right)^{m/2} n^{-m/2}
			\]
        \end{proof}
        Note that we have from above $2^{Dj_1}\leq n$ so this bound holds for any $j\leq j_1$.

        \begin{lemma}(Large Deviations)
            Let $X_1, \dots, X_n \sim p$ such that for a constant $c$, $\E_p|\psi_\lambda(X)|^2\leq c $. Let
                \begin{align*}
					\gamma_{\lambda}^p 
					    &= \E[\psi_\lambda(X)],\\
					\hat{\gamma}_{\lambda} 
					    &= \frac{1}{n}\sum_{i=1}^n \psi_\lambda(X_i),
				\end{align*}
            Let $l=\sqrt{j/n}$ and $\gamma>0$, then, for all $j$ s.t. $2^{Dj}\in o(n)$,
			    we have,
			    \begin{align*}
			    \Pr(|\hat{\gamma}_{\lambda}-\gamma_{\lambda}|>(K/2)l)
					&\leq 2 \times 2^{-\gamma n l^2}
			     \end{align*}
			where $K$ large enough such that 
		     \[
		            \frac{K^2}
					{8(c+\norm{\psi_\epsilon}_\infty(K/3))}> \log 2\gamma
		       \]
        \end{lemma}
        \begin{proof}
            Applying Bernstein's inequality we have 
			\begin{align*}	                         \Pr(|\hat{\gamma}_{\lambda}-\gamma_{\lambda}|>(K/2)l)
					&\leq 
						2 \exp \left(-\frac{n(K/2)^2l^2}
						{2(c+2^{D j/2}\norm{\psi_\epsilon}_\infty(K/3)l)}\right)\\
					&\leq 
						2 \exp \left(-\frac{K^2 n l^2}
						{8(\L_g+\norm{\psi_\epsilon}_\infty(K/3))}\right)
			\end{align*}
		    This implies for $K$ satisfying the above condition, 
			     \begin{align*}
    			    \Pr(|\hat{\gamma}_{\lambda}-\gamma_{\lambda}|>(K/2)l)
    					&\leq
    					2 \times 2^{ \left(- \gamma n l^2\right)}
			     \end{align*}
        \end{proof}

        Now for every $j\leq j_1$, $l$ satisfies the requirements of the above lemma. 
        So if $n l^2 (=j) \rightarrow\infty$ as $n\rightarrow \infty$ the probability of large deviation goes to zero.

        \begin{lemma}({\bf Variance})
            \label{lemma:variance}
            Let $X_1,\dots, X_n \sim p$ where $p$ is compactly supported, such that for a constant $c$, $\E_p|\psi_\lambda(X)|^m\leq c 2^{Dj(m/2-1)}$. Let $\F_d = B^{\sigma_d}_{p_d,q_d}$, then the variance of a linear wavelet estimator $\hat{p}$ with $j_0$ terms i.e. 
                \[
                    \hat{p}_n 
		        =   \sum_{ k\in \Z} \hat{\alpha}_{k} \phi_{k}+
		            \sum_{j=0}^{j_0} \sum_{\lambda\in \Lambda_j} \hat{\beta}_{\lambda}\psi_{\lambda}
                \]
            is bounded by 
                \[
                    d_{\F_d}(\hat{p}_n,\E[\hat{p}_n]) \leq c
                    \left(     \frac{1}{\sqrt{n}}+\frac{2^{j_0(D/2-\sigma_d)}}{\sqrt{n}}
                    \right)
                \]
            where $c = c_{p_d'} \left(\E_p|\psi_\lambda(X)|^2\right)^{1/2}$ is a constant.
        \end{lemma}
        \begin{proof}
            Since $\F_d = B^{\sigma_d}_{p_d,q_d}$ and $p$ is compactly supported we can, by lemma \ref{lemma:coeff} upper bound
                \[
                    \mathrm{E}_{X_1, \dots, X_n}
					\sup_{f\in \F_d}
					   	\sum_{k\in \Z}
						\alpha^f_k
						\left(
						\alpha_k^p- \hat{\alpha}_k\right)+
					\mathrm{E}_{X_1, \dots, X_n}
					\sup_{f\in \F_d}
						\sum_{j=0}^{j_0}\sum_{\lambda\in \Lambda_j}
							\beta^f_{\lambda}
							\left(
							\beta^p_{\lambda}- \hat{\beta}_{\lambda}\right)
                \]

            Since, for a constant $c$, $\E_p|\psi_\lambda(X)|^m\leq c 2^{Dj(m/2-1)}$ we can apply the moment bound below.
            For the first term  we have, (taking $\gamma=\alpha$ and $n_1=n_2=0$ in lemma \ref{lemma:discriminator} above)
				\begin{align*}
					\mathrm{E}_{X_1, \dots, X_n}
					\sup_{f\in \F_d}
					   	&\sum_{k\in \Z}
							\alpha^f_k
							\left(
							\alpha_k^p- \hat{\alpha}_k\right)&&\\ 
						&\leq 
							L_D 
							\left(
							\sum_{k} 
							\mathrm{E}_{X_1, \dots, X_n} 
					 		\left|\alpha_k^p- \hat{\alpha}_k\right|^{p_d'}\right)^{1/p_d'} 
					 		&&\quad \text{(finitely many terms)}\\
						&\leq 
							c L_D \norm{p}_\infty
							\left(
								(T+A)n^{-p_d'/2}
							\right)^{1/p_d'} 
					 		&&\quad \text{(moment bound)}\\
						&\leq
							 c
							n^{-1/2}&&
				\end{align*}
					where we use the fact only finitely many of the $\alpha$s are non-zero because of the compactness of the support of the densities we consider and the compactness of the wavelets.
					Similarly taking $\gamma = \beta$, $n_1=0$, $n_2=j_0$ in lemms \ref{lemma:discriminator} we have, using the moment bound as above,
				    \begin{align*}
						 \mathrm{E}_{X_1, \dots, X_n}
						 \sup_{f\in \F_d}
								&\sum_{j=0}^{j_0}\sum_{\lambda\in \Lambda_j}
								\beta^f_{\lambda}
								\left(
								\beta^p_{\lambda}- \hat{\beta}_{\lambda}\right)\\
							&\leq                 c\norm{p}_\infty L_D\sum_{j=0}^{j_0}2^{-j(\sigma_d+D/2-D/p_d)}\left(2^{Dj}(T+A)n^{-p_d'/2}\right)^{1/p_d'}
								\\
							&\leq             L_D\sum_{j=0}^{j_0}2^{-j(\sigma_d+D/2-D/p_d)}2^{D j/p_d'}n^{-1/2}\\
						 	&\leq  
						 	    c L_D\norm{p}_\infty 
								\sum_{j=0}^{j_0}2^{j(D/2-\sigma_d)}
								n^{-1/2}\\
							&\leq 
							    c\norm{p}_\infty\begin{cases}
								2^{j_0(D/2-\sigma_d)}
								n^{-1/2}  &\sigma_d\leq D/2\\
								n^{-1/2} &\sigma_d>D/2
								\end{cases}
					\end{align*}
        \end{proof}
        
        \begin{lemma}({\bf Bias})
            \label{lemma:bias}
            Let $X_1,\dots, X_n \sim p$ where $p \in B^{\sigma_g}_{p_g,q_g}$ is compactly supported and $\sigma_g\geq D/p_g$, $\F_d = B^{\sigma_d}_{p_d,q_d}$. Then the bias of a linear wavelet estimator $\hat{p}$ with $j_0$ terms is bounded by 
                \[
                    d_{\F_d}(p,\E_p[\hat{p}_n])
                    \leq c 2^{-j_0(\sigma_d+\sigma_g-(D/p_g-D/p_d')_+)}
                \]
            where $c$ is a constant that depends on $p_d$ and $\norm{\psi}_m$.
        \end{lemma}
        \begin{proof}
            Since $p$ is compactly supported, by lemma \ref{lemma:coeff} we need to upper bound 
            \[
                \sup_{\beta\in \F_d} 
				\sum_{j\geq j_1}\sum_{\lambda\in \Lambda} 
				\beta^f_{\lambda}
				\beta^p_{\lambda}
            \]
            Using lemma \ref{lemma:discriminator} and the fact that $\sigma_g\geq D/p_g$
			\begin{align*}
				\sup_{\beta\in \F_d} 
				&\sum_{j\geq j_0}\sum_{\lambda\in \Lambda} 
				    \beta^f_{\lambda}
				    \beta^p_{\lambda}\\
				&\leq L_D
						\sum_{j\geq j_0}
						2^{-j(\sigma_d+D/2-D/p_d)}
						\norm{\beta^p}_{p_d'}\\
				&= 
					L_D \sum_{j\geq j_0}	\frac{2^{j(\sigma_g+D/2-D/p_g)}}{2^{j(\sigma_d+\sigma_g+D-D/p_d-D/p_g)}}2^{j(D/p_d'-D/p_g)_+}
						\norm{\beta^p}_{p_g} \\
				&\leq 
					L_D \sum_{j\geq j_0}	\frac{2^{j(D/p_d'-D/p_g)_+}}{2^{j(\sigma_d+\sigma_g+D/p_d'-D/p_g)}}
					\sup_{j\geq j_0}2^{j(\sigma_g+D/2-D/p_g)}\norm{\beta^p}_{p_g}\\
				&\leq     2^{-j_0(\sigma_d+\sigma_g-(D/p_g-D/p_d')_+)}\norm{p}^{\sigma_g}_{ p_g q_g}&(\sigma_g\geq D/p_g)\\
				&\leq c 2^{-j_0(\sigma_d+\sigma_g-(D/p_g-D/p_d')_+)}
			\end{align*}
						
        \end{proof}

		Using lemmas \ref{lemma:bias} and \ref{lemma:variance} we get the following upper bound on the bias and variance of the linear wavelet estimator.  
		
		\[
		    c\left(
		        n^{-1/2} + n^{-1/2}2^{j_0(D/2-\sigma_d)}+2^{-j_0(\sigma_g+\sigma_d-D/p_g+D-D/p_d)}
		    \right)
		\]
		which when minimized for $j_0$ gives, 
		\[
		    2^{j_0} = n^{1/(2\sigma_g+D+2D/p_d'-2D/p_g)}
		\]
		which implies an upper bound of 
		\[
		    \lesssim n^{-1/2}+n^{-\frac{\sigma_g+\sigma_d-D/p_g+D-D/p_d}{2\sigma_g+D+2D/p_d'-2D/p_g}}
		\]
		as desired. 
\section{Proof of the Lower Bound}

    In this section we prove our main lower bound i.e. Theorem \ref{thm:nonlinear_lower_bound} using Fano's lemma and the Varshamov Gilbert bound as summarized below. 
		\begin{lemma} (Fano's Lemma; Simplified Form of Theorem 2.5 of \cite{tsybakov2009introduction})
		
            Fix a family $\P$ of distributions over a sample space $\X$ and fix a pseudo-metric $\rho : \P \times \P \to [0,\infty]$ over $\P$. Suppose there exists a set $T \subseteq \P$ such that there is a $p_0\in T$ with $p\ll p_0$ $\forall p\in T$ and  
            \[s := \inf_{p,p' \in T} \rho(p,p') > 0
              \quad \text{ , } \quad
              \sup_{p \in T} D_{KL}(p,p_0)
              \leq \frac{\log |T|}{16},\]
            where $D_{KL} : \P \times \P \to [0,\infty]$ denotes Kullback-Leibler divergence.
            Then,
            \[\inf_{\hat p} \sup_{p \in \P} \E \left[ \rho(p,\hat p) \right]
              \geq \frac{s}{16}\]
            where the $\inf$ is taken over all estimators $\hat p$. 
        \end{lemma}
        
        \begin{lemma} (Varshamov-Gilbert bound (\cite{tsybakov2009introduction}))
            Let $\Omega = \{0,1\}^m$ where $m\geq 8$. Then there exists a subset $\{w^0,\dots, w^M\}$ of $\Omega$ such that $w^0=(0,\dots,0)$ and 
            \[
                \omega(w^j,w^k)\geq \frac{m}{8} \quad \forall 0\leq j, k\leq M
            \]
            where $M\geq 2^{m/8}$, where $\omega(w^j, w^k) = \sum_{i=1}^m 1_{\{w^j_i\neq w^k_i\}}$ is the Hamming distance.
            \label{lemma:varshamov_gilbert}
        \end{lemma}
        
	    \begin{proof}(of Theorem \ref{thm:nonlinear_lower_bound})
	        We follow the method in Donoho et. al.  \cite{donoho1996density} and separate our proof into ``sparse'' and ``dense'' cases. As is standard procedure, for both cases we pick a finite subset of densities from $\F_g$ over which estimation is difficult. Since any function in a Besov space can be defined by its wavelet coefficients we pick a set of densities by an appropriate choice of wavelet coefficients. 
	        
	        Here we also need to pick a subset of functions from $\F_d$ so as to estimate $d_{\F_d}$. Following the method in \cite{singh2018minimax} we pick from $\F_d$, functions that are analogous to the ones we pick from $\F_g$ so that we measure the difference in the densities along the chosen perturbations. 
	        
	        We now fill in the details. We first let $g_0$ be a density function supported on an interval that contains $[-A,A]^D$ such that $\norm{g_0}_{\sigma_g p_g q_g}\leq L_G/2$ and $g_0 = c>0$ on $[-A,A]^D$.
	        
	        At a particular resolution $j$, we choose $2^{Dj}$ wavelets with disjoint supports; pick $\psi_{\lambda} = 2^{Dj/2}\psi_{\epsilon_1}(2^{Dj} x-k)$ indexed by $\lambda = 2^{-j}k + 2^{-(j+1)}\epsilon_1$ s.t. $k\in K_j$ where 
	            \[
			        K_j = \{-(2^j-1)A+2l A, l=0,\dots,(2^j-1)\}^D
			    \]
	        and $\epsilon_1=(1,0,\dots,0)$ (i.e. we pick the first wavelet). Note here that if $\lambda\neq \lambda'$ then $\psi_{\lambda}$ and $\psi_{\lambda'}$ have disjoint support.
	        
	        We now describe our choice of densities based on the set of coefficients $\zeta \subseteq \{\tau \in \Z^{|K_j|}:|\tau_\lambda|\leq 1\}$ i.e. 
	        \[
				\Omega_g := \{g_0+c_g\sum_{\lambda }\tau_{\lambda}\psi_{\lambda}:\tau \in \zeta, \lambda = 2^{-j}k+2^{-j-1}\epsilon_1, k\in K_j\} .
			\]
			
			If we pick $c_g$ to be small enough, every $p$ in $\Omega_g$ is a density function and is lower bounded on $[-A,A]^D$. Specifically if $c_g$ s.t.  
			    \[
			    c_g \leq
			        \frac{c}{2\norm{\psi}_{\infty}}2^{-D j/2}
		    	\]
			            
			then $\int g_0+c_g\sum_{\lambda}\tau_{\lambda}\psi_{\lambda} = 1$ (since $\int \psi_{\lambda}=0$) and,
				\[
					\norm{g_0-p}_{\infty} = c_g 2^{D j/2}\norm{\psi}_\infty \leq c/2
				\]
			so that $p$ is lower bounded on the domain of $\psi_{\lambda}$ by $c/2$ for every $\lambda$. This also implies that $p$ is always positive.
			 
			Now the following lemma states that if you have a small perturbation of a density s.t. the density is lower bounded on the support of the perturbation then the KL divergence between the perturbed and the original density is upper bounded by the $L^2$ norm of the perturbation. 
			\begin{lemma}
			    Let $g = g_0 + h , g_0$ be density functions such that $h\leq g_0$. If $S=\text{supp}(h) \subseteq \text{supp}(g)$ and $c\leq g$ on $S$, where $c$ is a constant. Then 
			        \[
			            D_{KL}(g^n, g_0^n) \leq 
			                c n \|g_0 - g\|^2_{\L^2}
			        \]
			\end{lemma}
			\begin{proof}
			    Since $g\leq 2g_0$ we have,
				\begin{align*}
					\frac{g_0-g}{g} &\geq -\frac{1}{2}
				\end{align*}
			so using the fact that $-\log(1+x)\leq x^2-x$ for all $x\geq -1/2$ we get
				\begin{align*}
					D_{KL}(g^n, g_0^n)
						& = n D_{KL}(g, g_0) \\
						& = 
						    n \int_S g(x) \log \frac{g(x)}{g_0(x)} \, dx \\
						& = 
						    -n \int_S g(x) \log \left( 1 + \frac{g_0(x) - g(x)}{g(x)} \right) \, dx \\
						& \leq 
						    n \int_S g(x) \left( \left( \frac{g_0(x) - g(x)}{g(x)} \right)^2 - \frac{g_0(x) - g(x)}{g(x)} \right) \, dx \\
						& = 
						    n \int_S \frac{\left( g_0(x) - g(x) \right)^2}{g(x)} \, dx
				\end{align*}
			which, since $g\geq c$ on $S$, is smaller than $c n \int_S \left( g_0(x) - g(x) \right)^2$ as desired. 
			\end{proof}
			
	        Using this fact we conclude that for any $p_\tau\in \Omega_g$, 
	        \[
	            KL(p_\tau, g_0) \leq n c_g^2 c \norm{\sum_{\lambda}\tau_\lambda\psi_\lambda}_{L^2}^2 = c n c_g^2 \norm{\tau}_{2}^2
	        \]
	        
	        Following the technique in \cite{singh2018adversarial} we also pick an analogous set of functions that live in $\F_d$ so that we can lower bound $d_{\F_D}$. In particular let 
				\[
					\Omega_d := \{c_d\sum_{\lambda}\tau_\lambda\psi_{\lambda}: \tau \in \zeta,\lambda = 2^{-j}k+2^{-j-1}\epsilon_1,k\in K_j\}
				\]
	        It now, only remains to choose appropriate sets $\zeta$ for the wavelet coefficients in each of the sparse and dense cases. In the remainder let $c$ be a constant not necessarily the same.

        \subsection*{Sparse or low-smoothness case:} 
            For the sparse/lower smoothness case we choose worst case densities to be perturbations along only a specific scaling of the wavelet at a time. In particular, let
				\[
				\zeta = \{ \tau : \tau_\lambda = 1, \tau_{\lambda'}=0, \lambda' \neq \lambda = 2^{-j}k+2^{-(j+1)}\epsilon_1, k \in K_j\}
				\]  
			We know from above that for any $c_g \leq c 2^{-Dj/2}$, every $p\in \Omega_g$ is a density such that $D_{KL}(p^n, g_0^n) \leq cn c_g^2 \norm{\tau}_2$.
			Now, we need
			\[
				\norm{g_0+c_g\psi_{\lambda}}^{\sigma_g}_{p_g q_g}\leq 
				\norm{{g_0}}^{\sigma_g}_{p_g q_g}+ 2^{j(\sigma_g+D/2-D/p_g)}c_g \leq L_g
			\]
			so that $\Omega_g\subseteq \F_g$. Since $\sigma_g\geq D/p_g$ the choice of  $c_g = c 2^{-j(\sigma_g+D/2-D/p_g)}$ suffices. 
			Similarly, $c_d = L_d 2^{-j(\sigma_d+D/2-D/p_d)}$ implies $\Omega_d\subseteq \F_d$.
			
			Then we pick $j$ large enough such that the KL divergence between any $p_\tau$ and $g_0$ is small. This enables us to apply Fano's lemma from above and get a lower bound. 
			
			So we need $c n c_g^2\leq \frac{\log |\zeta|}{16} = \frac{\log |K_j|}{16}$ i.e. 
			\[
			    n \leq c j/c_g^2 \iff n \leq 2^{2j(\sigma_g+D/2-D/p_g)}j 
			\]
			for the KL divergence to be small. 
			Given such a $j$ we have,
            \begin{align*}
				d_{\F_d}(p_\lambda, p_{\lambda'}) 
					&\geq \sup_{f\in \Omega_d} \left|\int c_g(f(x)(\psi_{\lambda}-\psi_{\lambda'})dx\right| = \norm{\psi}_{L^2}^2c_g c_d
			\end{align*}
			(since, $\norm{\psi_{\lambda}}_{L^2}^2 = \norm{\psi}_{L^2}^2$).
			So, if $2^j = (n/\log n )^{\frac{1}{2\sigma_g+D-2D/p_g}}$ we have, 
			\[
				M(\F_g,\F_d) \gtrsim \left(\frac{\log n}{n}\right)^{\frac{\sigma_g+\sigma_d+D-D/p_g-D/p_d}{2\sigma_g+D-2D/p_g}}
			\]

		\subsection*{\bf Dense or higher smoothness case:}
		    In the dense case, we choose our set of densities by perturbing $g_0$ along every scaling of the wavelet simultaneously i.e. let
		        \[
		            \zeta = \{ \tau : \tau_\lambda \in \{-1,+1\}\}
		        \]
			Now, we need
			\[
				\norm{g_0+c_g\sum_{\lambda}\tau_\lambda\psi_{\lambda}}^{\sigma_g}_{p_g q_g}\leq 
				\norm{{g_0}}^{\sigma_g}_{p_g q_g}+ 2^{j(\sigma_g+D/2)}c_g \leq L_g
			\]
			so that $\Omega_g\subseteq \F_g$. The choice of  $c_g = c 2^{-j(\sigma_g+D/2)}$ suffices. 
			Similarly, $c_d = L_d 2^{-j(\sigma_d+D/2)}$ implies $\Omega_d\subseteq \F_d$.
			
			Now the Varshamov-Gilbert bound from above implies we can pick a subset of $\Omega_G$ with size at least $2^{|K_j|/8}$ such that $\omega(\tau_\lambda,\tau_{k'}) \geq |K_j|/8$ which gives,
			\begin{align*}
				d_{\F_d}(p_\lambda, p_{\lambda'}) 
					&= \sup_{f\in \Omega_d} \left|\int c_g(f(x)(\psi_{\lambda}-\psi_{\lambda'})dx\right| \\
					&= c_gc_d\omega(\tau_\lambda, \tau_{\lambda'}) \geq c_g c_d \frac{2^{D j}}{4}
			\end{align*}

			We pick $j$ large enough such that the KL divergence between any $p_\tau$ and $g_0$ is small. This enables us to apply Fano's lemma from above and get a lower bound. In particular we need, for any $p_\tau \in \Omega_g$, $D_{KL}(p_\tau^n, g_0^n)\leq  c n c_g^2 \norm{\tau}_{2} =  c n c_g^2 |K_j|$ to be at most $\frac{\log |\zeta|}{16} =\frac{|K_j|}{16}$ which is equivalent to $n \leq 2^{j(2\sigma_g+D)}$. Then by Fano's lemma the lower bound in the dense case is 
			\[
				n^{-\frac{\sigma_g+\sigma_d}{2\sigma_g+D}}
			\]
            We combine the above two cases to get the following lower bound on the rate 
            \[
                \gtrsim \max{(n^{-\frac{\sigma_g+\sigma_d}{2\sigma_g+D}},n^{-\frac{\sigma_g+\sigma_d+D-D/p_g-D/p_d}{2\sigma_g+D-2D/p_g}})}
            \]
	
	\end{proof}

\section{Proof of the Upper Bound}
    We use the wavelet thresholding estimate as introduced in \cite{donoho1996density} to get an upper bound on our minimax rate.
    
    \begin{proof} (of theorem \ref{thm:nonlinear_upper_bound})
        We first upper bound our error by three terms namely, the stochastic error, the bias and the non-linear terms. The stochastic error is bounded above as usual by the above moment bound. The bias is bounded above by virtue of our density belonging to the besov space $B^{\sigma_g}_{p_g,q_g}$. The non-linear terms are more delicate. We follow the procedure in \cite{donoho1996density} and split them into four groups the first two of which are shown to be negligible as the probability of large deviations falls exponentially rapidly from Bernstein's inequality above. We simplify the upper bounds on the other two terms considerably by paying a penalty on the rate by the factor that is logarithmic in the sample size. We now fill in the details of the proof. 
    
        We first let our discriminator and generator classes be 
        \begin{align*}
            \F_d &= \{f : \norm{f}^{\sigma_d}_{p_d, q_d}\leq L_d\}\\
            \F_g &= \{p : \norm{p}^{\sigma_g}_{p_g, q_g}\leq L_g\}\cap \P\\
            \P &= \{p:p\geq 0, \norm{p}_{\L^1}=1, \text{supp}(p)\subseteq [-T,T]\}
        \end{align*}

		Given $X_1, \dots, X_n$ be IID with density $ p\in \F_g$ and the thresholded wavelet estimator of $p$ i.e. 
		\begin{align*}
		    p
			&=   \sum_{ k\in \Z} \alpha^p_{k} \phi_{k}+
			    \sum_{j\geq 0} \sum_{\lambda\in \Lambda_j} \beta^p_{\lambda}\psi_{\lambda}\\
			\hat{p}_n 
			&=   \sum_{ k\in \Z} \hat{\alpha}_{k} \phi_{k}+
			    \sum_{j=0}^{j_0} \sum_{\lambda\in \Lambda_j} \hat{\beta}_{\lambda}\psi_{\lambda}+
			    \sum_{j= j_0}^{j_1}\sum_{\lambda\in \Lambda_j} \tilde{\beta}_{\lambda} \psi_{\lambda}
		\end{align*}
		where 
			\[
			    \begin{aligned}[c]
			        \alpha^p_{k} &=  \E_{X\sim p}[\phi_{k}(X)]\\
			        \beta^p_{\lambda} &= \E_{X\sim p}[\psi_{\lambda}(X)]
			    \end{aligned}\hspace{3ex}\qquad
			    \begin{aligned}[c]
			        \hat{\alpha}_{k} &= \frac{1}{n}\sum_{i=1}^n \phi_{k}(X_i)\\
			        \hat{\beta}_{\lambda} &= \frac{1}{n}\sum_{i=1}^n \psi_{\lambda}(X_i)
			    \end{aligned}
			    \begin{aligned}[r]
			        \tilde{\beta}_{\lambda} &= \hat{\beta}_{\lambda}\mathbf{1}_{\{\hat{\beta}_{\lambda}>t\}}\\
			    \end{aligned}
			 \]
		with $t = K\sqrt{j/n}$, where $K$ is a constant to be specified later, and 
			\begin{align*}
			    2^{j_0} &= n^{\frac{1}{2\sigma_g+D}}\\
			    2^{j_1} &= n^{\frac{1}{2\sigma_g+D-2D/p_g}}
			\end{align*}
		we can upper bound the error as, 
			\begin{align*}
				d_{\F_d}(p,\hat{p}_n) 
					\leq \hspace{5ex}
						&\sup_{f\in \F_d}\sum_{k\in \Z}
						\alpha^f_{k}
						\left(\alpha_{ k}^p- \hat{\alpha}_{k}\right) 
						&+&
						\sup_{f\in \F_d}
						\sum_{j= 0}^{j_0}\sum_{\lambda\in \Lambda_j} 
						\beta^f_{\lambda}
						\left( \beta_{\lambda}^p- \hat{\beta}_{\lambda}\right)\\
						+&
						\sup_{f\in \F_d}
						\sum_{j\geq j_0}^{j_1}\sum_{\lambda\in \Lambda_j} 
						\beta^f_{\lambda}
						\left( \beta_{\lambda}^p- \tilde{\beta}_{\lambda}\right)
						&+& 
						\sup_{f\in \F_d}
						\sum_{j\geq j_1}\sum_{\lambda\in \Lambda_j} 
						\beta^f_{\lambda}
						 \beta_{\lambda}^p
			\end{align*}
			where the first three terms constitute the stochastic error (the non-linear terms or thresholded terms are also called `detail' terms \citep{donoho1996density}) and the last term is the bias. In particular:
        
			\begin{enumerate}
				\item 
				    The first term in our upper bound of the risk is the stochastic error or the variance of a linear wavelet estimator with $j_0$ terms. Note that since $\sigma_g\geq D/p_g$ $p\in \F_g$ implies by lemma \ref{lemma:besov_ub} that $\norm{p}_\infty<\infty$.
				    Then by substitution 
                    \[
                        \E_p|\psi_\lambda(X)|^{p_d'}\leq 
                            2^{-Dj (p_d'/2-1)}
                    \]
				    Therefore by lemma \ref{lemma:variance} we have an upper bound here of 
				    \[
				        cn^{-1/2} (2^{j_0(D/2-\sigma_d)}+1)
							\lesssim n^{-\frac{\sigma_g+\sigma_d}{2\sigma_g+D}}  +n^{-1/2}
				    \]

				\item 
				    The third term is the bias of a linear wavelet estimator with $j_1$ terms which by lemma \ref{lemma:bias} for $p_d'\geq p_g$ is bounded above by 
				    \[
				        c 2^{-j_1(\sigma_d+\sigma_g-D/p_g+D/p_d')}
				        \lesssim n^{-\frac{\sigma_g+\sigma_d+D-D/p_g-D/p_d}{2\sigma_g +D-2D/p_g}}
				    \]

				\item For the second term we have, by lemmas \ref{lemma:coeff} and \ref{lemma:discriminator}
					\begin{align*}
						\E\sup_{f\in \F_d} 
						\sum_{j\geq j_0}\sum_{\lambda\in \Lambda} 
						\beta^f_{\lambda}
						\left( \beta^p_{\lambda}- \tilde{\beta}_{\lambda}\right)
						&\leq
								L_D 
								\sum_{j= j_0}^{j_1} 
									2^{-j(\sigma_d+D/2-D/p_d)}
								\left(
									\E 
									\sum_{\lambda \in\Lambda_j}
										|\beta^p_{\lambda}- \tilde{\beta}_{\lambda}|^{p_d'}\mathrm{1}_A
								\right)^{1/p_d'}\\
							&\leq
								L_D 
								\sum_{j= j_0}^{j_1} 
									2^{-j(\sigma_d+D/2-D/p_d)}
								\left(
									\sum_{\lambda \in\Lambda_j}
									\E 
										|\beta^p_{\lambda}- \tilde{\beta}_{\lambda}|^{p_d'}\mathrm{1}_A
								\right)^{1/p_d'}
					\end{align*}
					where we are only summing over finitely many terms. 
					The set $A$ is given by the following cases:
					\vspace{2ex}
					
                    (For the upper bounds of the first two cases we have chosen $\gamma$ (which in turn determines the value of $K$) to be large enough so that the exponent of $2^j$ is negative and thus we can upper bound the geometric series by a constant multiple of the first term.) 
					\begin{enumerate}
						\item Let $A$ be the set of $k$
							s.t. $\hat{\beta}_{\lambda}>t$ and $\beta^p_{\lambda}<t/2$ and $r\geq 1/p_d'$ then 
							\begin{align*}
									L_D 
									\sum_{j= j_0}^{j_1} 
										&2^{-j(\sigma_d+D/2-D/p_d)}
									\left(
										\sum_{\lambda \in\Lambda_j}
										\E 
											|\beta^p_{\lambda}- \tilde{\beta}_{\lambda}|^{p_d'}\mathrm{1}_A
									\right)^{1/p_d'}\\
									&\leq 
										L_D 
										\sum_{j= j_0}^{j_1} 
										2^{-j(\sigma_d+D/2-D/p_d)}
										\left(
											\sum_{\lambda \in\Lambda_j}
											(\E 
												|\beta^p_{\lambda}- \tilde{\beta}_{\lambda}|^{p_d'r})^{1/r}\Pr(A)^{1/r'}
										\right)^{1/p_d'}\\
					        \end{align*}
					        Using the large deviation and moment bound
					        \[
					            \Pr(A)\leq 
					                \Pr\left(
					                    |\hat{\beta}_{\lambda}-\beta^p_{\lambda}|\geq t/2
					                \right) \leq c 2^{-\gamma j}
					        \]
					        we get,
					        \begin{align*}
									&\leq 
										c 
										\sum_{j= j_0}^{j_1} 2^{-j(\sigma_d+D/2-D/p_d)}
										\left( 
											2^{D j} 
												n^{-p_d'/2}
												2^{-j\gamma/r'}
										\right)^{1/p_d'}\\
									&\leq 
										c 
										\sum_{j= j_0}^{j_1} 
											2^{-j(\sigma_d+D/2-D/p_d-D/p_d')}
												n^{-1/2}
												2^{-\gamma j/p_d'r'}\\
									&\leq 
										c n^{-1/2}
										\sum_{j= j_0}^{j_1} 2^{-j(\sigma_d-D/2+\gamma /p_d'r')}\\
									&\leq 
										c n^{-1/2}
										2^{-j_0(\sigma_d-D/2+\gamma /p_d'r')}\\
									&\lesssim
									    n^{-\frac{\sigma_g+\sigma_d+\gamma/p_d'r'}{2\sigma_g+D}},
								\end{align*}
							which is negligible compared to the linear term.
						\item Let $B$ be the set of $k$
							s.t. $\hat{\beta}_{\lambda}<t$ and $\beta^p_{\lambda}>2t$ then same as above
							\begin{align*}
								\E\sup_{f\in \F_d} 
								\sum_{j= j_0}^{j_1}\sum_{\lambda \in\Lambda_j}  \beta^f_{\lambda}
								\beta^p_{\lambda}\mathrm{1}_B
								&\leq
									L_D 
									\sum_{j=j_0}^{j_1} 2^{-j(\sigma_d+D/2-D/p_d)}
									\norm{\beta_\lambda^p}_{p_d'}
									(\text{Pr}(B))^{1/p_d'}\\
								&\leq
									L_D 
									\sum_{j=j_0}^{j_1} 2^{-j(\sigma_d+D/2-D/p_d)}
									\norm{\beta_\lambda^p}_{p_d'}
									2^{-\gamma j/p_d'}\\
								&\leq
									L_D 
									\sum_{j=j_0}^{j_1} 
									2^{-j(\sigma_d+\sigma_g'+\gamma/p_d')}
									\sup_{j_0\leq j\leq j_1} 2^{j(\sigma_g'+D/2-D/p_d')}\norm{\beta_\lambda}_{p_d'}\\
								&\leq L_DL_G \sum_{j=j_0}^{j_1} 
									2^{- j(\sigma_d+\sigma_g'+\gamma/p_d')}\\
								&\leq L_DL_G 
									C2^{- j_0(\sigma_d+\sigma_g'+\gamma/p_d')}\\
								&\lesssim
								    n^{-\frac{\sigma_d+\sigma_g'+\gamma}{2\sigma_g+D}}
							\end{align*}
						which is negligible compared to the bias term. 

						\item Let $C$ be the set of $k$
							s.t. $\hat{\beta}_{\lambda}>t$ and $\beta^p_{\lambda}>t/2$ then:

							\begin{align*}
								\E\sup_{f\in \F_d} 
								\sum_{j= j_0}^{j_1}&\sum_{\lambda \in\Lambda_j} 
								\beta^f_{\lambda}
								\left( \beta^p_{\lambda}- \tilde{\beta}_{\lambda}\right)\mathrm{1}_C\\
								&\leq 
									L_D 
									\sum_{j= j_0}^{j_1} 
										2^{-j(\sigma_d+D/2-D/p_d)}
										\left( 
											\sum_{k\in C}
												\E
												|\beta^p_{\lambda}- \tilde{\beta}_{\lambda}|^{p_d'}
										\right)^{1/p_d'}\\
								&\leq 
									L_D 
									\sum_{j= j_0}^{j_1} C n^{-1/2}
									2^{-j(\sigma_d+D/2-D/p_d)}
									\left( 
										 \sum_{k\in C}
											\left(
											\frac{2\beta^p_{\lambda}\sqrt{n/j}}{K}
											\right)^{p_g}
									\right)^{1/p_d'}\\
								&\leq 
									L_D 
									\sum_{j= j_0}^{j_1} C n^{-1/2}(\sqrt{n/j})^{p_g/p_d'}2^{-j(\sigma_d+D/2-D/p_d)}
									\norm{\beta^p}_{p_g}^{p_g/p_d'}\\
								&\leq 
									L_D 
									\sum_{j= j_0}^{j_1} C n^{-1/2}(\sqrt{n/j})^{p_g/p_d'} 
									2^{-j(\sigma_d+D/2-D/p_d)}
									2^{-j(\sigma_g +D/2-D/p_g)p_g/p_d'}\\
									&\sup_{j_0\leq j\leq j_1}
									\norm{\beta}_{p_g}2^{j(\sigma_g +D/2-D/p_g)}\\
								&\leq 
									CL_D L_G n^{1/2(p_g/p_d'-1)}
									\sum_{j= j_0}^{j_1}  
									2^{-j((\sigma_g +D/2)p_g/p_d'+\sigma_d-D/2)}j^{-p_g/2p_d'}\\
								&\leq 
									CL_D L_G n^{1/2(p_g/p_d'-1)} 
									2^{-j_m((\sigma_g +D/2)p_g/p_d'+\sigma_d-D/2)}
							\end{align*}
							where 
							\[
								j_m = 
									\begin{cases}
										j_0 & (2\sigma_g+D)p_g \geq (D-2\sigma_d)p_d'\\
										j_1 & (2\sigma_g+D)p_g \leq (D-2\sigma_d)p_d'
									\end{cases}
							\]
							In the first case we have an upper bound of 
							\begin{align*}
								\lesssim  n^{-\frac{\sigma_g+\sigma_d}{2\sigma_g+D}}
							\end{align*}
							and in the second case we have an upper bound of 
							\[
								\lesssim
									n^{-\frac{\sigma_g+\sigma_d+D-D/p_d-D/p_g}{2\sigma_g +D-2D/p_g}} 
							\]

						\item Let $E$ be the set of $k$
							s.t. $\hat{\beta}_{\lambda}<t$ and $\beta^p_{\lambda}<2t$ then:
							\begin{align*}
								\E\sup_{f\in \F_d} 
								\sum_{j= j_0}^{j_1}&\sum_{\lambda \in\Lambda_j} 
								\beta^f_{\lambda}
								\beta^p_{\lambda}\mathrm{1}_D\\
								&\leq
									L_D 
									\sum_{j= j_0}^{j_1} 
									2^{-j(\sigma_d+D/2-D/p_d)}
									\left(
									\sum_{\lambda\in \Lambda_j} |\beta_{\lambda}^p|^{p_d'}\right)^{ 1/p_d'}\\
								&\leq
									L_D 
									\sum_{j= j_0}^{j_1} 
									2^{-j(\sigma_d+D/2-D/p_d)}
									\left(
									\sum_{\lambda\in \Lambda_j} |\beta_{\lambda}^p|^{p_g}(2t)^{p_d'-p_g}\right)^{ 1/p_d'} & p_d'\geq p_g\\
								&=
									L_D 
									\sum_{j= j_0}^{j_1} 
									2^{-j(\sigma_d+D/2-D/p_d)}(2t)^{1-p_g/p_d'}
									\norm{\beta}^{ p_g/p_d'} \\
								&\leq
									L_D 
									\sum_{j= j_0}^{j_1} 
									2^{-j(\sigma_d+D/2-D/p_d)}(2\sqrt{j/n})^{1-p_g/p_d'} 2^{-j(\sigma_g+D/2-D/p_g)p_g/p_d'}
									L_g \\
								&\leq
									c \sqrt{j_1} 
									n^{1/2(p_g/p_d'-1)}
									\sum_{j=j_0}^{j_1}2^{-j((\sigma_g +D/2)p_g/p_d'+\sigma_d-D/2)}j^{-p_g/2p_d'} \\
								&\lesssim \left(n^{-\frac{\sigma_g+\sigma_d}{2\sigma_g+1}}+n^{-\frac{\sigma_g+\sigma_d+D-D/p_d-D/p_g}{2\sigma_g +D-2D/p_g}}\right)\sqrt{\log n}
							\end{align*}
					\end{enumerate}
				\end{enumerate}
	\end{proof}
    
\section{Proof of Theorem \ref{thm:linear_minimax_rate}}
    \subsection*{Lower Bound}
        \begin{proof}
            Just as in the proof of the lower bound above we let $j\geq 0$ and 
                \[
					\Omega_g := \{g_0\pm c_g\psi_{\lambda}:\lambda = 2^{-j}k+2^{-j-1}\epsilon_1, k\in K_j\} 
				\]
			where $\epsilon_1=(1,0,\dots,0)$. Here we let $g_0 = 2^{Dj}c$ on at least $[-A,A]^D$ and 
			    \[
			        c_g = \min\left(
			        \frac{c}{2\norm{\psi}_{\infty}}2^{-D j/2}, \frac{L_g}{2}2^{-j(\sigma_g+D/2-D/p_g)}\right)
			    \]
			such that $\Omega_g\subseteq \F_g$.
			We also let 
			    \[
					\Omega_d := \{c_d\sum_{\lambda}\tau_\lambda\psi_{\lambda}:\lambda = 2^{-j}k+2^{-j-1}\epsilon_1, k\in K_j,\norm{\tau}\leq L_d \}
				\]
			s.t. 
				\begin{align*}
					c_d \leq L_d 2^{-j(\sigma_d+D/2-1/p_d)}
				\end{align*}
			i.e. $\Omega_d\subseteq \F_d$. 
			
			Then for any linear estimate $\hat{P}$ with  $\hat{\alpha}_\lambda = \int \psi_\lambda(x)d\hat{P}(x)$, 
			\begin{align*}
			    \sup_{P\in \F_g} &\E_P\sup_{f\in \F_d} \left|\int f(x)(dP(x)-d\hat{P}(x))\right|\\
			        &\geq \sup_{p\in \Omega_g} \E_P
			            \sup_{f\in \Omega_d}   \left|\int f(x)(p(x)dx-d\hat{P}(x))\right|\\
			        &= \sup_{\lambda:k\in K_j} \frac{c_d}{2}
			     \E_{g_0+c_g\psi_\lambda}
			        \left(\sup_{\tau:\norm{\tau}_{p_d}\leq L_d}
			            \sum_{\lambda'\neq \lambda} |\tau_{\lambda'}\hat{\alpha}_{\lambda'}|+ |\tau_{\lambda}||c_g-\hat{\alpha}_\lambda|
			            \right)\\
			            &+\E_{g_0-c_g\psi_\lambda}\left(\sup_{\tau:\norm{\tau}_{p_d}\leq L_d}
			            \sum_{\lambda'\neq \lambda} |\tau_{\lambda'}\hat{\alpha}_{\lambda'}|+ |\tau_{\lambda}||c_g-\hat{\alpha}_\lambda|
			            \right)\\
			        &\geq 
			            \sup_{\lambda:k\in K_j}\frac{c_d}{2}
			            \\&\sup_{\tau:\norm{\tau}_{p_d}\leq L_d}
			            \left(
    			            \sum_{\lambda'\neq \lambda} \E_{g_0+c_g\psi_\lambda}|\tau_{\lambda'}||\hat{\alpha}_{\lambda'}| +\E_{g_0-c_g\psi_\lambda}|\tau_{\lambda'}||\hat{\alpha}_{\lambda'}|+  \E_{g_0+c_g\psi_\lambda}|\tau_{\lambda}||c_g-\hat{\alpha}_\lambda|
    			            + \E_{g_0-c_g\psi_\lambda}|\tau_{\lambda}||c_g-\hat{\alpha}_\lambda|
			            \right)\\
			        &= 
			            \sup_{\lambda:k\in K_j}\frac{c_d}{2}
			            \\&
			            \left(
    			            \sum_{\lambda'\neq \lambda} (\E_{g_0+c_g\psi_\lambda}|\hat{\alpha}_{\lambda'}|)^{p_d'}+
    			            (\E_{g_0-c_g\psi_\lambda}|\hat{\alpha}_{\lambda'}|)^{p_d'}+
    			            (\E_{g_0+c_g\psi_\lambda}|c_g-\hat{\alpha}_\lambda|)^{p_d'}+
    			            (\E_{g_0-c_g\psi_\lambda}|c_g-\hat{\alpha}_\lambda|)^{p_d'}
			            \right)^{1/p_d'}\\
			          &\geq c_d
			            \left( \frac{1}{2^{Dj}}
    			            \sum_{\lambda'\neq \lambda} (\E_{g_0+c_g\psi_\lambda}|\hat{\alpha}_{\lambda'}|)^{p_d'}+
    			            (\E_{g_0-c_g\psi_\lambda}|\hat{\alpha}_{\lambda'}|)^{p_d'}+
    			            (\E_{g_0+c_g\psi_\lambda}|c_g-\hat{\alpha}_\lambda|)^{p_d'}+
    			            (\E_{g_0-c_g\psi_\lambda}|c_g-\hat{\alpha}_\lambda|)^{p_d'}
			            \right)^{1/p_d'}
			\end{align*}
			Now the expression inside the brackets is bounded below in \cite{donoho1996density} appendix A.3 by $n^{-1/2}2^{jD/p_d'}$ where $2^{j} = n^{\frac{1}{2\sigma_g-2D/p_g+2D/p_d'+D}}$ which implies a lower bound in our case of 
			\begin{align*}
			    &c2^{-j(\sigma_d+D/2-D/p_d)}n^{-1/2}2^{Dj/p_d'}\\
			    &= c2^{j(D/2-\sigma_d)}n^{-1/2}
			\end{align*}

		  which gives us a lower bound of 
		  \[
			    \gtrsim n^{-\frac{\sigma_d+\sigma_g-D/p_g+D/p_d'}{2\sigma_g-2D/p_g+2D/p_d'+D}}
			\]
			as desired.
                
        \end{proof}

\section{Proof of Theorem~\ref{thm:GAN_upper_bound}}
\label{app:GAN_upper_bound_proof}

Here, we prove the following theorem, which upper bounds the risk of an appropriately constructed GAN for learning Besov distributions:

\begin{theorem}[Convergence Rate of a Well-Optimized GAN]
Fix a Besov density class $B_{p_g,q_g}^{\sigma_g}$
    with $\sigma_g > D/p_g$ and discriminator class $B_{p_d,q_d}^{\sigma_d}$ with $\sigma_d>D/p_d$. Then, for any desired approximation error $\epsilon > 0$, one can construct a GAN $\hat p$ of the form~\eqref{eq:GAN_estimate} (with $\tilde{p}_n$)  with discriminator network $N_d \in \Phi(H_d,W_d,S_d,B_d)$ and generator network $N_g \in \Phi(H_g,W_g,S_g,B_g)$, s.t. for all $p \in B_{p_g,q_g}^{\sigma_g}$
    \begin{align*}
        \E \left[ d_{B_{p_d,q_d}^{\sigma_d}} \left( \hat{p}, p \right) \right] \lesssim
        \epsilon + 
        \E
        d_{B^{\sigma_d}_{p_d,q_d}}(\tilde{p}_n,p)
    \end{align*}
   where $H_d$, $H_g$ grow logarithmically with $1/\epsilon$, $W_d,S_d,B_d,W_g,S_g$, $B_g$ grow polynomially with $1/\epsilon$ and $C > 0$ is a constant that depends only on $B_{p_d,q_d}^{\sigma_d}$ and $B_{p_g,q_g}^{\sigma_g}$.
\label{thm:GAN_upper_bound_appendix}
\end{theorem}

Our statistical guarantees rely on a recent construction, by \citet{suzuki2018adaptivity}, of a fully-connected ReLU network that approximates Besov functions. Specifically, we leverage the following result:

\begin{lemma}[Proposition 1 of \citet{suzuki2018adaptivity}]
Suppose that $p,q,r \in (0,\infty]$ and $\sigma > \delta:= D(1/p - 1/r)_+$ and let $\nu = (\sigma - \delta)/(2\delta)$. Then, for sufficiently small $\epsilon \in (0,1)$, there exists a constant $C > 0$, depending only on $D,p,q,r,\sigma$, such that, for some

\[
    H \leq C \log (1/\epsilon), \quad 
    W \leq C\epsilon^{-D/\sigma}, \quad 
    S \leq C\epsilon^{-D/\sigma} \log (1/\epsilon),  \quad
    B \leq C\epsilon^{-(D/\nu+1)(1 \vee (D/p - \sigma)_+)/\sigma},
\]

$\Phi(H,W,S,B) \subseteq B_{p,q}^\sigma(1)$ and
$\Phi(H,W,S,B)$ approximates $B_{p,q}^\sigma(1)$ to accuracy $\epsilon$ in $\L^r$; i.e.,
\vspace{-2mm}
\[
    \sup_{f \in B_{p,q}^\sigma(1)} \inf_{f \in \Phi(H,W,S,B)} \|f - \tilde f\|_{\L^r} \leq C \epsilon.
\]
\label{lemma:suzuki}
\end{lemma}

\begin{proof}
\citet[Inequality 2.2]{liang2017well} showed that we can decompose the error, for densities $\hat{p}$, $p$,
\begin{align*}
    d_{\F_d} \left( \hat p, p \right)
    & \leq \inf_{q \in \Phi(H_g,W_g,S_g,B_g)} d_{\F_d} \left( p, q \right) \\
    & + 2 \sup_{f \in \F_d} \inf_{g \in \Phi(H_d,W_d,S_d,B_d)} \|f - g\|_\infty \\
    & + d_{\Phi(H_d,W_d,S_d,B_d)} \left( p, \tilde{p}_n \right) + d_{\F_d} \left( p, \tilde{p}_n \right),
\end{align*}
where the $3$ summands above correspond respectively the error of approximating $\F_g$ by $\Phi(L_g,W_g,S_g,B_g)$ (generator approximation error), the error of approximating $\F_d$ by $\Phi(L_d,W_d,S_d,B_d)$ (discriminator approximation error), and statistical error.

To bound the first term, note also that, since we assumed $\sigma_d > D/p_d$, we have the embedding $B_{p_d,q_d}^{\sigma_d} \subseteq \L^\infty$, and, in particular, $M := \sup_{f \in B_{p_d,q_d}^{\sigma_d}} \|f\|_{\L^\infty} < \infty$. Thus, by H\"older's inequality, the assumption that densities in $\P$ are supported only on $[-T,T]$, and Lemma~\ref{lemma:suzuki} (with $r = \infty$),
\begin{align*}
    \inf_{q \in \F_g} d_{\F_d} \left( p, q \right)
    \leq \inf_{q \in \F_g} \left( p, q \right) \sup_{f \in \F_D} \|f\|_{\L^1([-T,T])} \|p - q\|_{\L^\infty}
    \leq 2MT \epsilon.
\end{align*}
To bound the second term, simply observe that, by Lemma~\ref{lemma:suzuki}  (with $r = \infty$),
\[\sup_{f \in \F_d} \inf_{g \in \phi(L_g,W_g,S_g,B_g)} \|f - g\|_\infty \leq \epsilon.\]

Since, by Lemma~\ref{lemma:suzuki}, $\Phi(L_d,W_d,S_d,B_d) \subseteq B_{p_d,q_d}^{\sigma_d}$, the last term is immediately bounded (in expectation) by $d_{\F_d}(\tilde{p}_n, p)$.
Combining the bounds on these three terms gives
\[  
    d_{\F_d} \left( \hat{p}, p \right) \leq 2(MT + 1)\epsilon + 2
    d_{\F_d}(\tilde{p}_n, p).
\]
\end{proof}

{\small
  \bibliographystyle{plainnat}
  \bibliography{ref}
}

\end{document}